\documentclass{article}
\usepackage{amssymb}
\usepackage{verbatim}
\usepackage{array}
\usepackage{tikz}
\usepackage{latexsym}
\usepackage{enumerate}
\usepackage{amsmath}
\usepackage{amsfonts}
\usepackage{amsthm}
\usepackage{color}
\usepackage[english]{babel}
\usepackage{graphicx}
\usepackage[all]{xy}
\usepackage{boxedminipage}

\newtheorem{theorem}{Theorem}[section]
\newtheorem{proposition}[theorem]{Proposition}
\newtheorem{lemma}[theorem]{Lemma}
\newtheorem{corollary}[theorem]{Corollary}
\newtheorem{remark}[theorem]{Remark}
\newtheorem{definition}[theorem]{Definition}

\def\cH{\mathcal H}

\def\cX{\mathcal X}

\newcommand{\fqq}{\mathbb {F}_{q^2}}

\def\Aut{\mbox{\rm Aut}}
\def\K{\mathbb{F}_{q^{2n}}}

\def\Aut{\mbox{\rm Aut}}

\newcommand{\PGU}{\mbox{\rm PGU}}

\title{On subfields of the second generalization of the GK maximal function field}
\date{}
\author{Peter Beelen and Maria Montanucci}

\begin{document}
\maketitle

\begin{abstract}
The second generalized GK function fields $K_n$ are a recently found family of maximal function fields over the finite field with $q^{2n}$ elements, where $q$ is a prime power and $n \ge 1$ an odd integer. In this paper we construct many new maximal function fields by determining various Galois subfields of $K_n$. In case $\gcd(q+1,n)=1$ and either $q$ is even or $q \equiv 1 \pmod{4},$ we find a complete list of Galois subfields of $K_n.$ Our construction adds several previously unknown genera to the genus spectrum of maximal curves.
\end{abstract}

\thanks{}

\noindent
AMS: 11G20, 14H25, 14H37

\vspace{1ex}

\section{Introduction}

A function field $F$ defined over a finite field with square cardinality is called maximal, if the Hasse--Weil bound is attained. More precisely, for a function field $F$ of genus $g(F)$ over the finite field $\mathbb F_{Q^2}$ with $Q^2$ elements, the Hasse--Weil bound states that:
$$N(F)\leq Q^2+1+2g(F)Q,$$
where $N(F)$ denotes the number of rational places $N(F)$ of $F$. For a maximal function field it then holds that $N(F)\leq Q^2+1+2g(F)Q.$

An important example of a maximal function field is the Hermitian function field $H$ over the finite field $\mathbb F_{q^2}$. It can for example be defined as follows: $H=\mathbb F_{q^2}(x,y)$ with $y^{q+1}=x^{q+1}-1.$
The Hermitian curve has genus $q(q-1)/2$ (in fact the largest possible genus for a maximal function field over $\mathbb F_{q^2}$) and a large automorphism group isomorphic to $\PGU(3,q)$. Since a subfield of a maximal function field with the same field of constants is maximal by a theorem of Serre \cite{Lach}, computing fixed fields of $H$ of a subgroup of $\PGU(3,q)$ have given rise to many examples of maximal function fields. Since all maximal subgroups of $\PGU(3,q)$ are known, subgroups of these and the corresponding fixed fields have been studied in various papers, see for example \cite{AQ,BMXY,GSX,MZ}. One such maximal subgroup arises by considering the stabilizer of a rational place of $H$. This maximal subgroup, its subgroups and the corresponding fixed field of $H$ have for example been studied in \cite{BMXY,GSX}. Another maximal subgroup of $\PGU(3,q)$, which we will denote by $M_\ell$, arises by considering the stabilizer of a chord $\ell$. A full description of the subgroups of this group was given in \cite{DVMZ} for even $q$ and in \cite{MZq14} for $q \equiv 1 \pmod 4$. Many genera of maximal function fields have been obtained in this way, adding to the understanding of the genus spectrum of maximal curves. For $q \equiv 3 \pmod 4$ a complete list of subgroups is not known.

In \cite{GK} Giulietti and Korchm\'aros \cite{GK} introduced a new family of maximal function fields (GK function fields) over finite fields $\mathbb F_{q^6}$, which are not subfields of the Hermitian function field over the corresponding field for $q>2$. Therefore considering subfields of the GK function field, can give rise to new genera of maximal function fields. Such examples were found in for example \cite{FG}. Later, the GK function field was generalized in \cite{GGS} to a family of maximal function fields over finite fields $\mathbb F_{q^{2n}}$ with $n$ odd. These maximal function fields are often called the Garcia--G\"{u}neri--Stichtenoth (GGS) function fields. All subgroups of the automorphism groups of these fields were classified in \cite{ABB}, but before that several subfields were already determined in \cite{GOS}.

Recently a second generalization of the GK function field was discovered \cite{BM}. As for the GGS function field, for each odd $n$ a maximal function field $K_n$ was found with constant field $\mathbb F_{q^{2n}}.$ Though the genus of $K_n$ is equal to the corresponding GGS function field, their automorphism groups are different. A preliminary study in \cite{BM} already revealed that new genera of maximal function fields can be obtained by considering fixed fields of subgroups of $\Aut(K_n).$ The current article expands upon these results and is the analogue of \cite{ABB} for the second generalization $K_n$ of the GK function field. We are mainly interested in the cases that $q$ is even or $q \equiv 1 \pmod q$, since otherwise not even a complete list of subgroups of $M_\ell$ is known. With this restriction, we construct many subgroups of the automorphism group of $K_n$. If additionally $\gcd(q+1,n)=1$, the list is complete.

\section{The curve $\cX_n$ and its automorphism group}

Throughout this paper $p$ is a prime, $q=p^h$ with $h \geq 1$ and $n \geq 1$ is odd.  Let $\mathbb{K}$ be the algebraic closure of $\K$. We denote with $\cX_n \subset \mathbb{P}^3$ the algebraic curve defined by the following affine equations

$$\cX_n :  \begin{cases}   \displaystyle Y^{q+1}=X^{q+1}-1\\ \displaystyle Z^m=Y \frac{X^{q^2}-X}{X^{q+1}-1}\end{cases}, $$

where $m:=(q^n+1)/(q+1)$. Further, let $K_n=\mathbb{K}(x,y,z)$ with $y^{q+1}=x^{q+1}-1$ and  $z^m=y( x^{q^2}-x)/(x^{q+1}-1)$ be the function field of $\cX_n$ over $\mathbb{K}$. The Hermitian curve $\cH_q \subset \mathbb{P}^2$ defined by the affine equation $Y^{q+1}=X^{q+1}-1$, gives rise to the subfield $H:=\mathbb{K}(x,y)$ of $K_n$.

The curve $\cX_n$ and its function field $K_n$ were constructed and studied in \cite{BM}, where it was shown that $\cX_n$ is a maximal curve when considered over the finite field $\K$. Moreover, the full automorphism group of $K_n$, and hence $\cX_n$, was determined in $\cite{BM}$. More precisely, it was shown there that $$Aut(K_n)=\{\alpha_{a,b,\xi} \mid a^{q+1}-c^{q+1}=1,\, \xi^{q^n+1}=1\},$$ with $\alpha_{a,b,\xi}$ acting on $x,y,z \in K_n$ as follows:
$$\left(\begin{array}{c}\alpha_{a,b,\xi}(x) \\ \alpha_{a,b,\xi}(y) \\ \alpha_{a,b,\xi}(z) \end{array}\right):=
\left(\begin{array}{ccc} a & c^q \xi^m & 0 \\ c & a^q\xi^m & 0 \\ 0 & 0 & \xi\end{array}\right)\left(\begin{array}{c}x \\ y \\ z \end{array}\right).$$
In particular, we have $|Aut(K_n)|=q(q^2-1)(q^n+1).$ Denote by $O$ the set of $q^3+1$ places of $K_n$ corresponding to the $\mathbb{F}_{q^2}$-rational points of $\cX_n$. The group $Aut(K_n)$ acts on $O$ in two orbits. One orbit, which we denote by $O_1$, is formed by the $q+1$ places $R_{\infty}^1, \ldots, R_{\infty}^{q+1}$ of $K_n$, centered at the $q+1$ points at infinity of $\cX_n$. We denote the other orbit $O\setminus O_1$ with $O_2$. Denoting by $Aut(K_n)_P$ the subgroup of automorphisms fixing a place $P$ of $K_n$, we see from the Orbit Stabilizer Theorem that  $|Aut(K_n)_P|=q(q^2-1)m$ for every $P \in O_1$, while $|Aut(K_n)_P|=q^n+1$ for every $P \in O_2$.
From \cite[Theorem 11.49]{HKT}, we deduce that for $P \in O_1$ we have $Aut(\cX_n)_P \cong E_q \rtimes C_{(q^2-1)m}$, where $E_q$ is an elementary abelian group of order $q$ while $C_{(q^2-1)m}$ is cyclic of order $(q^2-1)m$. Similarly for $P \in O_2$, we have $Aut(\cX_n)_P \cong C_{q^n+1}$ is cyclic. The orbits $O_1$ and $O_2$ are the only two short orbits of $Aut(K_n).$

Any automorphism of $K_n$ gives rise to an automorphism of $H$ by restriction. Hence we have a group homomorphism $\pi: Aut(K_n) \to Aut(H).$ The restriction $\pi(\alpha_{a,b,\xi}) \in Aut(H)$ will be denoted by $\beta_{a,b,\xi^m}.$ Note that $\beta_{a,b,\xi^m}$ only depends on $a,b$ and $\xi^m$, justifying the notation. Note that $M_\ell:=\pi(Aut(K_n))$ is a maximal subgroup of $Aut(H) \cong \PGU(3,q)$ of cardinality $(q^3-q)(q+1).$ The notation $M_\ell$ is motivated by the fact that this group consists exactly of those elements of $\PGU(3,q)$ that stabilize the line $\ell$ defined by $T=0$ when using the homogeneous coordinates $(X:Y:T)$ for $\mathbb{P}^2$. We use several subgroups of $Aut(K_n)$ very frequently in the sequel; see \cite{BM} for more details. First of all, define $S_\ell:=\{ \alpha_{a,b,1} \mid a^{q+1}-c^{q+1}=1\}$. It is a normal subgroup of $Aut(K_n)$ of cardinality $q^3-q$ isomorphic to ${\rm{SL}}(2,q).$ Since $\pi(S_\ell) \cong S_\ell$, we may identify these groups. Therefore, with slight abuse of notation we will denote them both by $S_\ell$. In this way, group $S_\ell$ can be interpreted as a normal subgroup of $M_\ell$ of index $q+1$. A further subgroup that we will use frequently, is $C_{m}:=\ker \pi=\{\alpha_{1,0,\xi} \mid \xi^{m}=1\}.$ Since $C_{m}=\ker \pi$, we have $Aut(K_n)/C_m \cong M_\ell.$ Note that $C_m$ is a central subgroup of $Aut(K_n)$ as well as a cyclic group of order $m$. As a result, the subgroup of $Aut(K_n)$ generated by $S_\ell$ and $C_m$ can be written as a direct product $S_\ell \times C_m$. Also note that any non-trivial element from $C_m$ fixes all $q^3+1$ places of $K_n$, corresponding to the $\fqq$-rational points of $\cX_n.$ Moreover, the fixed field of $C_m$ is precisely the function field $H$. The group $Aut(K_n)$ also contains a cyclic group $C_{q^n+1}$ of order $q^n+1$, namely $C_{q^n+1}=\{\alpha_{1,0,\xi} \mid \xi^{q^n+1}=1.\}$ Note that $C_{q^n+1} \cap S_\ell=\{\alpha_{1,0,1}\}.$ Since $S_\ell$ can be seen as a normal subgroup of $Aut(K_n),$ this gives rise to the following semidirect product description: $Aut(K_n) \cong S_\ell \rtimes C_{q^n+1}.$

Given a subgroup $L \subset Aut(K_n)$, we obtain by restriction a subgroup $\pi(L)$ of $Aut(H).$ To easy the notation, we write $\bar L:=\pi(L)$. From the Second Isomorphism Theorem, we obtain that $\bar L \cong L/(L\cap \ker \pi)\cong LC_m/C_m$. In the sequel, we will write $L_m:=L \cap C_m.$ In particular, we have $|L|=|\bar L|\cdot |L_m|.$ As a first result, we relate the genera of the fixed field of a subgroup $L \subset Aut(K_n)$ with that of the subgroup $\bar L \subset Aut(H).$ This relation will give the first step in reducing the classification of the genera of Galois subfields of $K_n$ to that of the genera of Galois subfields of $H$ containing the fixed field of $M_\ell$. For a subgroup $L \subset Aut(K_n),$ we denote with $K_n^L$ its fixed field and similarly $H^{\bar L}$ denotes the fixed field of $\bar L \subset M_\ell.$ Further a subgroup of $Aut(K_n)$ or $M_\ell$ is called tame, if it has order relatively prime to $q$.

\begin{theorem}\label{thm:genusrelations}
Let $L \subset Aut(K_n)$ be a subgroup and write $L_m=L \cap C_m$ and $\bar L=\pi(L).$ Further assume that the set $O$ of $q^3+1$ places corresponding to $\cX_n(\fqq)$ is partitioned in $N$ orbits under the action of $L$. Then $$2 g_L-2=\frac{m}{|L_m|}\left(2g_{\bar L}-2\right)+N\left( \frac{m}{|L_m|}-1 \right),$$
where $g_L$ (resp. $g_{\bar L}$) denotes the genus of $K_n^L$ (resp. $H^{\bar L}$). In case $\bar L$ is tame, we have $$g_{L}=g_{\bar L}+\frac{(q^2-1)(q+1)(m/|L_m|-1)}{2|\bar L|}.$$
\end{theorem}
\begin{proof}
First of all note that any place of $H^{\bar L}$ ramified in the extension $K_n/H^{\bar L}$ needs to lie below a place in $O$, since $O_1$ and $O_2$ are the only short orbits of $Aut(K_n).$ Moreover, since $H=K_n^{C_m},$ the fixed field of $C_m$, we see that $H^{\bar L}=K_n^{LC_m}.$ Therefore, the extension $K_n^L/H^{\bar L}$ is a Galois extension with cyclic Galois group $LC_m/L \cong C_m/L_m.$ Since any element of $C_m$ fixes all places in $O$, any place of $H^{\bar L}$ lying below a place in $O$ is totally ramified in the extension $K_n^L/H^{\bar L}.$ Hence the number of places of $H^{\bar L}$ that ramify in the extension $K_n/H^{\bar L}$ is equal to $N$. Since the extension degree $[K_n:H^{\bar L}]=m/|L_m|$ is tame, the Riemann-Hurwitz theorem implies the first part of the theorem.

Now suppose that $\bar L$ is a tame subgroup. Let $\bar O$ denote the set of places of $H$ corresponding to the $\fqq$-rational points of $H_q$. The above proof implies in particular that the places in $O=O_1 \cup O_2$ are totally ramified in $K_n/H$. This implies that the action of the subgroup $L$ on $O$ is equivalent to the one of $\bar L$ on $\bar O$. In particular, the number of orbits in $\bar O$ under the action of $\bar L$ is the same as $N$, the number of orbits in $O$ under the action of $L$. For a given place $P$ of $H$ denote its restriction to $H^{\bar L}$ with $P_{\bar L}$. Moreover, let $e(P|P_{\bar L})$ denote the ramification index of $P$ in the extension $H/H^{\bar L}.$ Then we have
$$\sum_{P \in \bar O} e(P|P_{\bar L}) = \sum_{P_{\bar L}} \sum_{P|P_{\bar L}} e(P|P_{\bar L})=\sum_{P_{\bar L}} [H:H^{\bar L}]=N\cdot |\bar L|.$$
Here the summation $\sum_{P_{\bar L}}$ is over all places $P_{\bar L}$ of $H^{\bar L}$ lying below a place in $\bar O.$  On the other hand, since $\bar L$ is tame, the Riemann-Hurwitz theorem applied to the extension $H/H^{\bar L}$ implies that
$$\sum_{P \in \bar O} e(P|P_{\bar L})=|O|+\sum_{P \in O} e(P|P_{\bar L})-1 = |O|+q^2-q-2-|\bar L|(2g_{\bar L}-2).$$
Here we used that $H$ has genus $q(q-1)/2$. Combining these expressions and using that $|O|=q^3+1$, we can express $N$ in terms of $q$, $|\bar L|$ and $g_{\bar L}$. Substituting this expression in the formula given in the first part of the theorem, we obtain the desired result.
\end{proof}

\begin{remark}
The proof of Theorem \ref{thm:genusrelations} implies that $N$ can be computed given only $\bar L$, since $N$ is equal to the number of orbits in $\bar O$, the set of places of $H$ corresponding to the $\fqq$-rational points of $\mathcal{H}_q$, under the action of $\bar L.$ Note that the group $M_\ell$ acts on the places of $H$ with two short orbits $\bar O_1$ and $\bar O_2$ given by the places lying below those in $O_1$ and $O_2$ respectively.
\end{remark}

The above remark shows that once a group $\bar L \subset M_\ell$ is given, the number of orbits $N$ can be determined. The only data from $L$ that is needed in order to compute $g_L$ is that cardinality of $L_m=L \cap C_m.$ Our strategy in the classification of all possible genera $g_L$ is to classify all possible genera $g_{\bar L}$ and number of orbits $N$ for subgroups $\bar L \subset M_\ell$, and then to study what the possibilities for $|L_m|$ are for a given $\bar L \subset M_\ell$. In order to do this, we study certain subgroups of $L$. We first define two maps used to analyze the situation.

\begin{definition}
For a divisor $k$ of $q^n+1$, let $\mu_k \subset \K^*$ denote the multiplicative cyclic subgroup of $\K^*$ of order $k$. Define $\rho: Aut(K_n) \to \mu_{q^n+1}$ by $\rho(\alpha_{a,b,\xi})=\xi$. Similarly define $\bar\rho: Aut(H) \to \mu_{q+1}$ by $\bar\rho(\beta_{a,b,\zeta})=\zeta.$
\end{definition}

\begin{lemma}\label{lem:Ltotriple}
Let a subgroup $L \subset Aut(K_n)$ be given and write $L_0:=\rho(L)$, $L_1:=L \cap (S_\ell \times C_m)$ and $\bar L:=\pi(L)$. Then the following hold:
\begin{enumerate}
\item $\bar\rho(\bar L)=L_0^m,$
\item $\bar L \cap S_\ell=\pi(L_1),$
\item $\rho(L_1)=L_0 \cap \mu_m.$
\end{enumerate}
\end{lemma}
\begin{proof}
To prove the first item, note that $\bar \rho(\pi(\alpha_{a,b,\xi}))=\xi^m=\rho(\alpha_{a,b,\xi})^m$. Since $\bar L= \pi(L),$ this implies that $\bar\rho(\bar L)=L_0^m.$
The second item follows, first observe that $\pi(L \cap (S_\ell \times C_m)) = \pi(L) \cap S_\ell,$ since $\pi(S_\ell \times C_m)=S_\ell$ and $\pi^{-1}(S_\ell)=S_\ell \times C_m.$
This implies that $\pi(L_1) = \pi(L \cap (S_\ell \times C_m)) = \pi(L) \cap S_\ell =\bar L \cap S_\ell,$ proving the second item of the lemma.
Finally, if $\alpha_{a,b,\xi} \in L_1,$ then $\xi^m=1,$ since $L_1 \subset S_\ell \times C_m.$ Hence $\rho(L_1) \subseteq L_0 \cap \mu_m.$ On the other hand, if $\xi \in L_0 \cap \mu_m,$ then there exists $\alpha_{a,b,\xi} \in L$ such that $\xi^m=1.$ By definition, this implies that $\alpha_{a,b,\xi} \in L_1.$ Hence $\rho(L_1) \supseteq L_0 \cap \mu_m.$
\end{proof}

The point of this lemma is that a group $L$ gives rise to a triple of groups $(L_0,L_1,\bar L)$, with certain properties that at least partially determine $L$. 
The following theorem gives a partial converse.

\begin{theorem}\label{thm:tripletoL}
Let $L_0 \subset \mu_{q^n+1},$ $L_1 \subset S_\ell \times C_m$ and $\bar L \subset M_\ell$ be groups satisfying $\bar\rho(\bar L)=L_0^m,$ $\bar L \cap S_\ell=\pi(L_1),$ and $\rho(L_1)=L_0 \cap \mu_m.$ Moreover, assume that $L_1 \cap C_m=\{\alpha_{1,0,\xi} \mid \xi \in L_0 \cap \mu_m\}.$ Then there exists a subgroup $L \subset Aut(K_n)$ such that $\rho(L)=L_0$, $L \cap (S_\ell \times C_m)=L_1$, and $\pi(L)=\bar L.$
\end{theorem}
\begin{proof}
Let a triple $(L_0,L_1,\bar L)$ with the indicated properties be given and denote with $\eta$ a generator of the cyclic group $L_0$. Then $\zeta:=\eta^m$ is a generator of $L_0^m$. Further define $s:=|L_0^m|$. Since $\bar \rho (\bar L)=L_0^m$, for each integer $i$, there exists an element $g_i \in Aut(K_n)$ such that $\pi(g_i) \in \bar L$ and $\bar \rho (\pi(g_i)) =\zeta^i$. We set $g_0=\alpha_{1,0,1}$ to be the identify element of $Aut(K_n)$. Now consider the set
$$L:=\cup_{i=0}^{s-1} g_i L_1 \subseteq Aut(K_n).$$

First of all, we claim that $L$ is a subgroup of $Aut(K_n)$. Indeed, choose two elements from $L$, say $g_i l_1$ and $g_j l_2$, where $l_1,l_2 \in L_1$. Further choose an integer $k$ such that $0 \le k < s$ and $i-j \equiv k \pmod{s}.$ To show that $L$ is a group, all we need to show is that the element $g:=g_k^{-1}(g_i l_1) (g_j l_2)^{-1}$ is an element of $L_1$, since then $(g_i l_1)(g_j l_2)^{-1} \in g_k L_1 \subseteq L.$ Now note that $\pi(g) \in \bar L$, since $\pi(L_1) \subseteq \bar L$ and $\pi(g_i),\pi(g_j),\pi(g_k) \in \bar L.$ On the other hand by choice of $k$, $\bar \rho (\pi(g))=\zeta^{i-j-k}=1,$ with together with the previous implies that $\pi(g) \in \bar L \cap S_\ell.$ Since $\bar L \cap S_\ell \subseteq \pi(L_1),$ we see that there exists $h \in L_1$ such that $\pi(g)=\pi(h)$, implying that $gh^{-1} \in \ker \pi = C_m$ and hence that $gh^{-1}=\alpha_{1,0,\xi}$ for some $\xi \in \mu_m.$ On the other hand, $\xi=\rho(gh^{-1})=\eta^{i-j-k}\rho(l_1l_2^{-1}h^{-1}) \in L_0.$ Combined with the previous and the assumption that $L_1 \cap C_m = \{\alpha_{1,0,\xi} \mid \xi \in L_0 \cap \mu_m\}$, we see that $gh^{-1} \in L_1.$ Since by construction $h \in L_1$, this implies that $g \in L_1$ just as we wanted to show. This concludes the proof of the claim that $L$ is a subgroup of $Aut(K_n)$.

It is clear by construction that $\rho(L) \subseteq L_0,$ since $\rho(g_i)=\eta^i \in L_0$ and $\rho(L_1) \subseteq L_0.$ Moreover, since $\rho(g_1)=\eta$ is a generator of $L_0$, we see that $\rho(L)=L_0.$ Now we show that $L \cap (S_\ell \times C_m)=L_1$. It is trivial to see that the inclusion $L \cap (S_\ell \times C_m) \supseteq L_1$ holds. Now suppose that $g_i l_1 \in L\cap (S_\ell \times C_m)$, where $g_i$ is as in the previous and $l_1 \in L_1.$ If $g_il_1 \in S_l \times C_m$, then $\bar \rho(\pi(g_il_1))=1,$ but on the other hand $\bar \rho(\pi(g_il_1))=\zeta^i$. Since $0 \le i < s$ and $\zeta$ has order $s$, this implies that $i=0.$ But then $g_il_1=l_1 \in L_1,$ which is what we wanted to show. Finally we prove that $\pi(L)=\bar L.$ From the construction of $L$ it is clear that $\pi(L) \subset \bar L$ and that $\pi(L)=\cup_{i=0}^{s-1}\pi(g_i)\pi(L_1)=\cup_{i=0}^{s-1}\pi(g_i)(\bar L \cap S_\ell).$ Since $\bar \rho(\pi(g_i))=\zeta^i,$ we have $|\pi(L)|=s\cdot|\bar L \cap S_\ell|.$ On the other hand, the map $\bar \rho$ restricted to $\bar L$ has image $L_0^m$ and kernel precisely $\bar L \cap S_\ell.$ Therefore, we obtain $|\bar L|=s|\bar L \cap S_\ell|.$ This implies $\bar L = \pi(L).$
\end{proof}

\begin{corollary}\label{cor:tupletoL}
Let $L_0 \subset \mu_{q^n+1}$ be a subgroup of cardinality $r$ such that $L_0^m$ is a group of cardinality $s$. Then $r=s \cdot \gcd(r,m)$. Moreover, let $\bar L \subset M_\ell$ be a group satisfying $\bar\rho(\bar L)=L_0^m,$ then there exists a subgroup $L \subseteq M_\ell$ such that $\pi(L)=\bar L$, $\rho(L)=L_0,$ and $|L \cap C_m|=\gcd(r,m).$
\end{corollary}
\begin{proof}
It is easy to see that if $L_0$ is a cyclic group of order $m$, then $L_0^m$ has cardinality $m/\gcd(r,m).$ This proves the first statement. Since $S_\ell$ can be interpreted as a subgroup of $Aut(K_n),$ the same is true for $\bar L \cap S_\ell.$ Also the group $L_0 \cap \mu_m$ can be interpreted as a subgroup of $Aut(K_n)$ by identifying $\xi \in L_0 \cap \mu_m$ with $\alpha_{1,0,\xi}.$ With these identifications in mind, define $L_1:=(\bar L \cap S_\ell) \times (L_0 \cap \mu_m) \subset S_\ell \times C_m.$ Clearly $\bar L \cap S_\ell=\pi(L_1),$ and $\rho(L_1)=L_0 \cap \mu_m.$ Moreover, $|L \cap C_m|=|L_1 \cap C_m|=|L_0 \cap \mu_m|=|L_0|/|L_0^m|=\gcd(r,m)$ by the first part. Theorem \ref{thm:tripletoL} can therefore be applied.
\end{proof}

In general, for a subgroup $L$ of $Aut(K_n)$, $L_m=L \cap C_m=L_1 \cap C_m$ can be thought of as a subgroup of $L_0 \cap \mu_m$, but it need not be the whole group. For the group $L$ constructed in the proof of Theorem \ref{thm:tripletoL}, we have $L_m \cong L_0 \cap \mu_m.$ Other groups $L$ giving rise to the triple may exist such that $L_m$ and $L_0 \cap \mu_m$ do not have the same cardinality. In view of Theorem \ref{thm:genusrelations}, the cardinality of $L_m$ is important when calculating possible genera $g_L$ of the fixed field of $L$ and the groups constructed in Theorem \ref{thm:tripletoL} may not give rise to a complete list of possible genera $g_L.$ However, as we will show now in some cases all genera will already be obtained using the groups constructed in Theorem \ref{thm:tripletoL}.

\begin{corollary}
Let $\bar L \subset M_\ell$ be a subgroup and write $s:=|\bar\rho(\bar L)|.$ Further, let $L \subset Aut(K_n)$ be a subgroup such that $\pi(L)=\bar L$ and write $L_m=L \cap C_m$. If $\gcd(s,m/|L_m|)=1,$ then there exists a subgroup $\tilde L \subseteq Aut(K_n)$ constructed using Theorem \ref{thm:tripletoL} such that $g_L=g_{\tilde L}.$
\end{corollary}
\begin{proof}
Theorem \ref{thm:genusrelations} implies that the genus of $g_L$ depends on data involving $\bar L$, but otherwise only on the cardinality of $L_m.$ Now for $\tilde L_1:=(\bar L \cap S_\ell)\times L_m \subset S_\ell \times C_m,$ we have $|\rho(\tilde L_1)|=|L_m|.$ Let $L_0 \subseteq \mu_{q^n+1}$ be the cyclic subgroup of order $r:=s\cdot |L_m|.$ Then $\gcd(r,m)=|L_m|,$ since by assumption $\gcd(s,m/|L_m|)=1$ and hence $|L_0^m|=s.$ Applying Corollary \ref{cor:tupletoL} we obtain a group $\tilde L$ with the property that $\pi(\tilde L)=\bar L$, $\rho(\tilde L)=L_0,$ and $\tilde L \cap C_m=L_m.$ Hence $g_L=g_{\tilde L}$ by Theorem \ref{thm:genusrelations}, which is what we wanted to show.
\end{proof}

Note that the condition $\gcd(s,m/|L_m|)=1$ is automatically satisfied in case $\gcd(m,q+1)=1,$ since $s$ divides $q+1$. Hence we can classify all possible genera $g_L$ in terms of subgroups $\bar L$ of $M_\ell$ if $\gcd(n,q+1)=1$. If this condition is not satisfied, we obtain many, but potentially not all, possible genera $g_L.$ In the next section, we now turn our attention to a description of subgroups of $M_\ell$ when $q$ is even or $q \equiv 1 \pmod 4$ as well as their combinatorial dates needed to apply Theorem \ref{thm:genusrelations}.

\section{Some preliminary results on $M_\ell$ and $\PGU(3,q)$} \label{prel:Herm}

The Hermitian curve $\mathcal{H}_q$ can be seen as the set of points $P=(X:Y:Z)$ in $\mathbb{P}^2$  satisfying $Y^{q+1}=X^{q+1}-Z^{q+1}$, that is, as the set of isotropic points of $\mathbb{P}^2=\mathbb{P}^2(\mathbb{K})$, where $\mathbb{K}=\bar{\mathbb{F}}_{q^2}$, with respect to the unitary polarity defined by the Hermitian form $Y^{q+1}-X^{q+1}+Z^{q+1}$.

In particular, given a point $P=(a:b:c)$ in $\mathbb{P}^2$ then its polar line $\ell$ with respect to $\cH_q$ is defined as $\ell: -a^qX+b^qY+c^qZ=0$. If $P \in \cH_q$ then $\ell$ is the tangent line of $\cH_q$ at $P$, otherwise $P \not\in \ell$ and $\ell$ is $(q+1)$-secant line at $\cH_q$. The couple $(P,\ell)$ is called a pole-polar pair (with respect to $\cH_q$.

Using this geometrical point of view, the following lemma describes how an element in $\PGU(3,q)$ of a given order acts on $\mathbb{P}^2$, and in particular on $\mathcal{H}_q$. This can be obtained using the usual terminology of collineations of projective planes. In particular, a linear collineation $\sigma$ of $\mathbb{P}^2$ is a $(P,\ell)$-\emph{perspectivity} if $\sigma$ preserves  each line through the point $P$ (the \emph{center} of $\sigma$), and fixes each point on the line $\ell$ (the \emph{axis} of $\sigma$). A $(P,\ell)$-perspectivity is either an \emph{elation} or a \emph{homology} according to $P\in \ell$ or $P\notin\ell$, respectively. A $(P,\ell)$-perspectivity is in  $PGL(3,q^2)$ if and only if its center and its axis are in $\mathbb{P}^2(\mathbb{F}_{q^2})$.

\begin{lemma}\label{classificazione}{\rm (\!\!\cite[Lemma 2.2]{MZrs})}
For a nontrivial element $\sigma\in\PGU(3,q)$, one of the following cases holds.
\begin{itemize}
\item[(A)] ${\rm ord}(\sigma)\mid(q+1)$ and $\sigma$ is a homology, whose center $P$ is a point of $\mathbb{P}^2(\mathbb{F}_{q^2})\setminus\cH_q$ and whose axis $\ell$ is a chord of $\cH_q(\mathbb{F}_{q^2})$ such that $(P,\ell)$ is a pole-polar pair.
\item[(B)] $p\nmid{\rm ord}(\sigma)$ and $\sigma$ fixes the vertices $P_1$, $P_2$, $P_3$ of a non-degenerate triangle $T\subset \mathbb{P}^2(\mathbb{F}_{q^6})$.
\begin{itemize}
\item[(B1)] ${\rm ord}(\sigma)\mid(q+1)$, $P_1,P_2,P_3\in\mathbb{P}^2(\mathbb{F}_{q^2})\setminus\cH_q$, and $T$ is self-polar.
\item[(B2)] ${\rm ord}(\sigma)\mid(q^2-1)$, ${\rm ord}\nmid(q+1)$, $P_1\in\mathbb{P}^2(\mathbb{F}_{q^2})\setminus\cH_q$, and $P_2,P_3\in\cH_q(\mathbb{F}_{q^2})$. 
\item[(B3)] ${\rm ord}(\sigma)\mid(q^2-q+1)$, and $P_1,P_2,P_3\in\cH_q(\mathbb{F}_{q^6})\setminus\cH_q(\mathbb{F}_{q^2})$.
\end{itemize}
\item[(C)] ${\rm ord}(\sigma)=p$ and $\sigma$ is an elation, whose center $P$ is a point of $\cH_q(\mathbb{F}_{q^2})$ and whose axis $\ell$ is tangent to $\cH_q$ at $P$ such that $(P,\ell)$ is a pole-polar pair.
\item[(D)] either ${\rm ord}(\sigma)=p$ with $p\ne2$, or ${\rm ord}(\sigma)=4$ with $p=2$; $\sigma$ fixes a point $P\in\cH_q(\mathbb{F}_{q^2})$ and a line $\ell$ which is tangent to $\cH_q$ at $P$, such that $(P,\ell)$ is a pole-polar pair.
\item[(E)] ${\rm ord}(\sigma)=p\cdot d$, where $1\ne d\mid(q+1)$; $\sigma$ fixes two points $P\in\cH_q(\mathbb{F}_{q^2})$ and $Q\in\mathbb{P}^2(\mathbb{F}_{q^2})\setminus\cH_q$; $\sigma$ fixes the line $PQ$ which is the tangent to $\cH_q$ at $P$, and another line through $P$ which is the polar of $Q$.
\end{itemize}
\end{lemma}

In the following we will refer to an element $\sigma \in \PGU(3,q)$ to be of type (A), (B), (C), (D) or (E) as in Lemma \ref{classificazione}.

Let $P$ be the point with homogeneous coordinates $(0:0:1)$ and let $\ell: Z=0$ be the polar line of $P$. The maximal subgroup $M_\ell$ of $\PGU(3,q)$ fixing $P$ (equivalently $\ell$) has order $q(q^2-1)(q+1)$ and it is given by the following matrix representation
$$M_\ell=\Bigg{\{} \begin{pmatrix} a & \tau c^q & 0 \\ c & \tau a^q & 0\\ 0 & 0 & 1\end{pmatrix} : a,c,\tau \in \mathbb{F}_{q^2}, \ a^{q+1}-c^{q+1}=1, \ \tau^{q+1}=1 \Bigg\}.$$
An element $\sigma \in M_\ell$ will be identified with the triple $[a,c,\tau]$, see \cite{BM}.

The group $M_\ell$ has a normal subgroup of index $q+1$ given by
$$S_\ell=\Bigg{\{} \begin{pmatrix} a & c^q & 0 \\ c &  a^q & 0\\ 0 & 0 & 1\end{pmatrix} : a,c \in \mathbb{F}_{q^2}, \ a^{q+1}-c^{q+1}=1\Bigg\} \cong {\rm{SL}}(2,q).$$
The center of $M_\ell$ is given by
$$Z=Z(M_\ell)=\langle \alpha \rangle, \ {\rm where} \ \alpha= \begin{pmatrix} \epsilon & 0 & 0 \\ 0 &  \epsilon & 0\\ 0 & 0 & 1\end{pmatrix}=[\epsilon,0,\epsilon^2], \ {\rm and} \  \epsilon^{q+1}=1 \ {\rm primitive}.$$

\begin{remark} \label{center}
Note that $Z$ fixes $\bar O_1$ pointwise. Indeed, if $\bar P_1=(a:b:0) \in \ell$ then $\alpha(\bar P_1)=(\epsilon a: \epsilon b:0)=\bar P_1$. Furthermore, $M_\ell$ does not contain elements of type (B3) and (D) since the do not fix any point in $\mathbb{P}^2(\mathbb{F}_{q^2}) \setminus \cH_q$. If $\sigma \in M_\ell$ fixes a point $R \in \bar O_2$ then $\alpha$ is of type (A).
\end{remark}

If $p=2$ the group $M_\ell$ is isomorphic to ${\rm{SL}}(2,q) \times Z$ since $S_\ell \cap Z$ is trivial, see \cite{DVMZ}. If $p$ is odd then $S_\ell \cap Z=\langle [-1,0,1] \rangle$.
If $q$ is odd then $M_\ell$ can be written as $S_\ell \rtimes C_{q+1}$ where $C_{q+1}$ is generated by an element of type (A) whose center is an $\mathbb{F}_{q^2}$-rational point $Q \in \ell \setminus \cH_q$, see \cite{MZ}.
If $q \equiv 1 \pmod 4$ then defining $Z_1$ to be the subgroup of $Z$ of order $(q+1)/2$, then $Z_1 \cap S_\ell$ is trivial and $\langle S_\ell,Z_1\rangle=S_\ell \times Z_1$.

The complete list of subgroups of $M_\ell$ up to isomorphism is known for $p=2$ and $q \equiv 1 \pmod q$, see \cite{DVMZ} and \cite{MZq14} respectively. Since later, we will need these groups for a case by case analysis, we list them in the following two lemmas. A group $C_e$ will denote a cyclic group of order $e$. In the first of the two next lemmas, the groups $C_w$ can always be seen as a subgroup of $Z$.

\begin{lemma}{\rm{\cite{DVMZ}}} \label{sugroupsMlqeven}
Let $p=2$ and $q=2^h$ where $h \geq 1$ and let $w$ be an arbitrary divisor of $q+1$. The following is the complete list of subgroups of $M_\ell$ up to isomorphism.
\begin{enumerate}
\item $E_{2^f} \times C_w$, where $f \leq h$, $E_{2^f}$ is elementary abelian of order $2^f$;
\item ${\rm{SL}}(2,2) \times C_w$ where either $h$ is even or $3 \nmid w$;
\item $(C_{3^k} \rtimes C_2) \times C_{w/3^k}$, where $h$ is odd and $3^k || w$;
\item ${\rm{SL}}(2,2^f) \times C_{w}$, where $f>1$ and $f \mid h$;
\item $D_{2t} \times C_w$ where $D_{2t}$ is a dihedral group of order $2t$ with $t \mid (q-1)$;
\item $\mathbf{A}_5 \times C_w$ and $h$ is even;
\item $\mathbf{A}_4 \times C_w$ and $h$ is even;
\item $(E_{2^f} \rtimes C_d) \times C_w$, where $f \leq h$, $E_{2^f}$ is elementary abelian of order $2^f$ and $C_d$ is cyclic of order $d$ where $d\mid \gcd(2^f-1,q-1)$;
\item $C_d \times C_w$, where $d \mid (q-1)$;
\item groups fixing a self-polar triangle $T=\{P_1,P_2,P_3\} \subseteq \mathbb{P}^2(\mathbb{F}_{q^2}) \setminus \cH_q(\mathbb{F}_{q^2})$ fixing $P=P_1$ and acting transitively on $T \setminus \{P_1\}$;
\item groups fixing a self-polar triangle $T=\{P_1,P_2,P_3\} \subseteq \mathbb{P}^2(\mathbb{F}_{q^2}) \setminus \cH_q(\mathbb{F}_{q^2})$, with $P=P_1$ pointwise.
\end{enumerate}
Furthermore, if $\bar L \leq M_\ell$ is a subgroup of type 1-9 then $\bar L \cap Z$ is isomorphic to $C_w$.
\end{lemma}

In the next lemma we list the subgroups of $M_\ell$ in case $q \equiv 1 \pmod 4$. The mentiond groups $C_w$ can be identified with subgroups of $Z_1$.

\begin{lemma}{\rm{\cite{MZq14}}} \label{subqroupsMlq14}
Let $p$ an odd prime, $q=p^h$ with $h \geq 1$, $q \equiv 1 \pmod 4$ and let $w$ be an arbitrary divisor of $(q+1)/2$. The following is the complete list of subgroups of $M_\ell$ up to isomorphism.
\begin{enumerate}
\item ${\rm{SL}}(2,5) \times C_w$, when $q^2 \equiv 1 \pmod 5$;
\item $G_{48} \times C_w$  when $p \geq 5$, $8 \mid (q-1)$ and $G_{48}$ has order $48$;
\item ${\rm{SL}}(2,3) \times C_w$ when $p \geq 5$ and $3 \nmid w$;
\item $(Q_8 \rtimes C_{3^k}) \times C_{w/3^{k-1}}$ where $p \geq 5$, $k \geq 2$, $3^{k-1}||w$, and $Q_8$ is the quaternion group of order $8$;
\item$C_d$ of order $d \mid (q^2-1)$ with $d \nmid (q+1)$;
\item $Dic_d \times C_w $ where $Dic_d=\langle \delta,\epsilon \mid \delta^{2d}=1, \epsilon^2=\delta^d, \ \epsilon^{-1}\delta\epsilon=\delta^{-1}\rangle$ is a dicyclic group of order $4d$ and $1<d \mid (q-1)/2$;
\item  ${\rm{SL}}(2,p^k) \times C_w$ where $k \mid h$;
\item ${\rm{TL}}(2,p^k) \times C_w$ where $k \mid h$ and $h/k$ is even;
\item $({\rm{SL}}(2,3) \rtimes C_2) \times C_w \cong SmallGroup(48,29) \times C_w$ where $p \geq 5$ and $8 \nmid (q-1)$;
\item $D_{2d} \times C_w$ where $D_{2d}$ is a dihedral group of order $2d$ with $2<d \mid (q-1)$;
\item $\hat{Dic}_m \times C_w$ where $\hat{Dic}_m=Dic_m \rtimes C_2=\langle \alpha,\xi \mid \alpha^{4m}=1,\alpha^{2m}=\xi^2, \xi^{-1} \alpha \xi=\alpha^{2m-1} \rangle$ has order $8m$ and $m \mid (q-1)/2$ but $m \nmid (q-1)/4$;
\item ${\rm{SU}}^{\pm}(2,p^k) \times C_w \cong ({\rm{SL}}(2,p^k) \rtimes C_2)\times C_w$, where $k \mid h$ and $h/k$ is odd;
\item groups fixing a self-polar triangle $T=\{P_1,P_2,P_3\} \subseteq \mathbb{P}^2(\mathbb{F}_{q^2}) \setminus \cH_q(\mathbb{F}_{q^2})$, with $P=P_1$ pointwise;
\item  groups fixing a self-polar triangle $T=\{P_1,P_2,P_3\} \subseteq \mathbb{P}^2(\mathbb{F}_{q^2}) \setminus \cH_q(\mathbb{F}_{q^2})$ fixing $P=P_1$ and acting transitively on $T \setminus \{P_1\}$;
\item groups fixing a point $R \in \cH_q(\mathbb{F}_{q^2}) \cap \ell$.
\end{enumerate}
Furthermore if $\bar L \leq M_\ell$ is a subgroup of type 1-12 or of type 15 then $\bar L \cap Z_1$ is isomorphic to $C_w$.
\end{lemma}

With these preliminaries in place, we proceed by analyzing the cases $q$ even and $q \equiv 1 \pmod 4$ in the following two sections.

\section{The case $q$ even: determination of the number of orbits $N$}

Let $q=2^h$ with $h \geq 1$. The genus of $\cH_q/ \bar L$ with $\bar L \leq M_\ell$ was computed in \cite[Section 4]{DVMZ}. However, in case the characteristic divides the order of $\bar L,$ we also need to know the number $N$ of $\bar L$-orbits in $\bar O_1 \cup \bar O_2$ before being able to apply Theorem \ref{thm:genusrelations}. By revisiting the genus computations in \cite[Section 4]{DVMZ}, we achieve this in the current section. We will denote with $N_i$ the number of orbits of $\bar L$ in its action on $\bar O_i$ with $i=1,2$, so that $N=N_1+N_2$. In the following $\bar L_Z$ will denote $\bar L \cap Z$ and $w=|\bar L_{Z}|$.

We now proceed with a case-by-case analysis for $\bar L$ according to Lemma \ref{sugroupsMlqeven}. Note that if $\bar L$ is a group of order coprime with $p$ then we only need to know $g_{\bar L}$ in order to compute $N$ according to Theorem \ref{thm:genusrelations}. Therefore, we will not address Cases 9 and 11 from Lemma \ref{sugroupsMlqeven} in this section.

\begin{lemma} \label{elab1}
Let $L \leq Aut(\cX_n)$ and let $\pi(L)=\bar L=E_{2^f} \times C_w$, where $f \leq h$, $E_{2^f}$ is elementary abelian of order $2^f$ and $C_w=\bar L_Z$.
Then
$$g_{\bar L}=\frac{(q+1)(q-w-2^f)+w(2^f+1)}{2^{f+1}w},$$
and
$$N=\frac{q}{2^f}+1+\frac{q}{2^f} \cdot \frac{q^2-1}{w}.$$
\end{lemma}

\begin{proof}
The computation of $g_{\bar L}$ was given in \cite[Proposition 4.4]{DVMZ}. By Lemma \ref{classificazione} (C)  $\bar L$ fixes exactly one point $\bar P_1 \in \bar O_1$, $\bar L_Z=C_w$ fixes $\bar O_1$ pointwise from Remark \ref{center} and every other element in $\bar L$ has exactly $\bar P_1$ as its unique fixed point from Lemma \ref{classificazione} (E). This implies that $\bar L$ has an orbit of length $1$ in $\bar O_1$ and from the Orbit Stabilizer Theorem it acts on $\bar O_1 \setminus \{\bar P_1\}$ with orbits of length $2^f$, as $|\bar L_{\bar Q}|=|\bar L_Z|$ and $2^f=|\bar L|/|\bar L_Z|$ for every $\bar Q \in \bar O_1\setminus \{\bar P_1\}$. Also $\bar L$ acts with long orbits on $\bar O_2$ since no elements in $\bar L$ fix points in $\bar O_2$. This shows that
$N_1=1+\frac{(q+1)-1}{2^f}=\frac{q}{2^f}+1,$
while
$N_2=\frac{|\bar O_2|}{|\bar L|}=\frac{q^3-q}{2^f|\bar L_Z|}=\frac{q}{2^f} \cdot \frac{q^2-1}{w}.$
\end{proof}

\begin{lemma} \label{sl222}
Let $L \leq Aut(\cX_n)$ and let $\bar L={\rm{SL}}(2,2) \times C_w$ where $C_w=\bar L_Z$ and either $h$ is even or $3 \nmid w$. Then
$$g_{\bar L}=\begin{cases}  \displaystyle \frac{q^2-wq-3q+4w-4}{12w},  & \makebox{if $h$ is even,}
\\ \\  \displaystyle \frac{(q+1)(q-w-4)+9w}{12w},  & \makebox{otherwise.} \end{cases} \quad {\rm{and}} \quad N=\begin{cases}  \displaystyle \frac{q+8}{6}+\frac{q}{2} \cdot \frac{q-1}{3} \cdot \frac{q+1}{w},  & \makebox{if $h$ is even,} \\ \\  \displaystyle \frac{q+4}{6}+\frac{q^3-q}{6w}, & \makebox{otherwise.} \end{cases}$$

Let $\bar L=(C_{3^k} \rtimes C_2) \times C_{w/3^k}$, where $k \geq 1$, $C_w = \bar L_Z$, $h$ is odd and $3^k || w$. Then,
$$g_{\bar L}=\frac{(q+1)(q-w-8)+9w}{12w}, \quad {\rm{and}} \quad N=\frac{q+4}{6}+\frac{(q+1)(q^2-q-2)}{w}+\frac{q+1}{w}.$$
\end{lemma}

\begin{proof}
From \cite[Proposition 4.2]{DVMZ} subgroups $\bar L$ of types 2 and 3 in Lemma \ref{sugroupsMlqeven} are exactly those for which $\bar L/(\bar L \cap Z) \cong {\rm{SL}}(2,2)$. The genus $g_{\bar L}$ was already computed in \cite[Proposition 4.2]{DVMZ} and the action of $\bar L$ on $\bar O_1 \cup \bar O_2$ can be deduced from its computation. In fact, from the proof of \cite[Proposition 4.2]{DVMZ},
\begin{itemize}
\item If $h$ is even, that is, $3 \mid (q-1)$ then $\bar L$ acts on the two fixed points of its unique subgroup $D$ of order $3$ which is of type (B2) from Lemma \ref{classificazione}. Also these points are in $\bar O_1$ since they are points of $\ell$. The $3$ involutions of $\bar L$ each fix  a point of $\bar O_1$. since these involutions do not commute with the elements of order $3$ in $D$ these points are mutually distinct.

Thus $\bar L/\bar L_Z={\rm{SL}}(2,2)$ acts with orbits of length $6$ on the remaining points of $\bar O_1$. In this way we get that
$N_1=2+\frac{(q+1)-3-2}{6}=\frac{q+8}{6}.$
Since $\bar L$ contains no elements fixing a point in $\bar O_2$ we get also that
$N_2=\frac{q^3-q}{|\bar L|}=\frac{q}{2} \cdot \frac{q-1}{3} \cdot \frac{q+1}{w}.$
\item Let $h$ be odd and $3 \mid |\bar L_Z|$. The three involutions of $\bar L$ each fix a point in $\bar O_1$, while the stabilizer in $\bar L$ of one of the remaining points in $\bar O_1$ is $\bar L_Z$ since now $D$ is generated by an element of type (B1). Hence
$N_1=1+\frac{(q+1)-3}{6}=\frac{q+4}{6}.$
As recalled, the two elements of order $3$ in $\bar L$ are of type $(B1)$ and hence fix pointwise a self-polar triangle $T$ having $\ell$ as a side. From the proof of \cite[Proposition 4.2]{DVMZ}, $\bar L$ contains $2$ other subgroups  of order $3$ and hence a further $4$ elements of order $3$ which turn out to be of type (A). Each of the subgroups of order $3$ fix a different side of $T$. Also, these sides are both different from $\ell$. This implies that $\bar O_2$ contains a set of $2(q+1)$ points, the points of intersection of the two sides of $T$ with $\cH_q$, on which $\bar L$ acts with stabilizer of order $3$, and hence with orbits of length $2|\bar L_Z|$. Moreover from \cite[Proposition 4.2]{DVMZ} $\bar L$ acts with long orbits on the remaining $q^3-q-2(q+1)$ points of $\bar O_2$. This shows that
$N_2=\frac{2(q+1)}{2|\bar L_Z|}+\frac{q^3-q-2(q+1)}{|\bar L|}=\frac{q+1}{w}+\frac{(q+1)(q^2-q-2)}{6w}.$
\item Let $h$ be odd and $3 \nmid |\bar L_Z|$. As before, the three involutions of $\bar L$ fixes exactly $3$ distinct points which are in $\bar O_1$, while the stabilizer in $\bar L$ of one of the remaining points in $\bar O_1$ is $\bar L_Z$. Hence
$N_1=1+\frac{(q+1)-3}{6}=\frac{q+4}{6}.$
The difference in this case is that from the proof of \cite[Proposition 4.2]{DVMZ} no elements in $\bar L$ fix a point in $\bar O_2$, implying that
$N_2=\frac{q^3-q}{6|\bar L_Z|}=\frac{q^3-q}{6w}.$
\end{itemize}
\end{proof}

\begin{lemma} \label{SL2f6}
Let $L \leq Aut(\cX_n)$ and let $\bar L={\rm{SL}}(2,2^f) \times C_w$ with $C_w=\bar L_Z$, $f\mid h$ and $f>1$. Then
$$ g_{\bar L}=\frac{(q+1)\left[q-w-2^f(2^f-1)\gcd(2^f+1,w)-2^f\right]+(2^f+1)w(2^{2f}-2^f+1)}{2^{f+1}(2^f+1)(2^f-1)w},$$
if $h/f$ is odd, while
$$g_{\bar L}=\frac{(q+1)(q-2^{2f}-w)-w(2 \cdot 2^{3f}-2^{2f}-2 \cdot 2^{f}-1)}{2^{f+1}(2^f+1)(2^f-1)w}+1,$$
if $h/f$ is even. Also,
$$N=\begin{cases}  \displaystyle 1+\frac{q-2^f}{2^f(2^{2f}-1)}+\frac{(q+1)\gcd(w,2^f+1)}{(2^f+1)w}+\frac{q+1}{w} \cdot \frac{q(q-1)-2^f(2^f-1)}{2^f(2^f-1)(2^f+1)}, & \makebox{if  $h/f$ is odd,} \\ \\  \displaystyle 2+\frac{q-2^{2f}}{2^f(2^{2f}-1)} + \frac{q}{2^f} \cdot \frac{q-1}{2^{2f}-1} \cdot \frac{q+1}{w}, & \makebox{if h/f is even.} \end{cases}$$
\end{lemma}

\begin{proof}
The genus $g_{\bar L}$ was computed in \cite[Proposition 4.9]{DVMZ} and according to its computation we can determine the action of $\bar L$ on $\bar O_1 \cup \bar O_2$ according to $h/f$ odd or $h/f$ even.
\begin{itemize}
\item Let $h/f$ be odd. Then $2^f+1$ divides $q+1$. Since ${\rm{SL}}(2,2^f)$ contains exactly $2^f+1$ Sylow $2$-subgroups, it has an orbit of length $2^f+1$ on $\bar O_1$ given by the corresponding fixed points. The other elements in ${\rm{SL}}(2,2^f)$ have no fixed points on $\bar O_1 \cup \bar O_2$ and hence $\bar L$ acts with orbits of length $2^f(2^{2f}-1)$ on the remaining $(q+1)-(2^f+1)=q-2^f$ points in $\bar O_1$. Since $2^f+1$ divides $q+1$ it might be that $\gcd(2^f+1,|\bar L_Z|)$ is not trivial. From \cite[Proposition 4.9]{DVMZ} if this would happen then $\bar O_2$ would contain a subset of $2^f(2^f-1)(q+1)$ points whose stabilizer in $\bar L$ has order $\gcd(2^f+1,|\bar L_Z|)$. Hence on these points $\bar L$ would act with orbits of length $|\bar L|/\gcd(2^f+1,|\bar L_Z|)$. Also $\bar L$ acts with long orbits on the remaining $q^3-q-2^f(2^f-1)(q+1)$ points in $\bar O_2$. Hence
$N_1=1+\frac{q-2^f}{2^f(2^{2f}-1)},$ and $N_2=\frac{\gcd(2^f+1,|\bar L_Z|)2^f(2^f-1)(q+1)}{|\bar L|}+\frac{(q^3-q)-2^f(2^f-1)(q+1)}{|\bar L|}=\frac{(q+1)\gcd(w,2^f+1)}{(2^f+1)w}+\frac{q+1}{w} \cdot \frac{q(q-1)-2^f(2^f-1)}{2^f(2^f-1)(2^f+1)}$.

\item Let $h/f$ be even. Then $2^{2f}-1$ divides $q-1$ and hence all the elements of odd order in ${\rm{SL}}(2,2^f)$ are of type (B2) from Lemma \ref{classificazione}. Since ${\rm{SL}}(2,2^f)$ contains exactly $2^f+1$ Sylow $2$-subgroups, it has an orbit of length $2^f+1$ on $\bar O_1$ given by the corresponding fixed points. The elements of order $2^f+1$ fix two points on $\bar O_1$, and hence from the Orbit Stabilizer Theorem $\bar O_1$ contains also a set of $2^f(2^f-1)$ points whose stabilizer has order $2^f+1$. Also $\bar L$ acts semiregularly on $\bar O_2$ and with orbits of length $2^f(2^{2f}-1)$ on the remaining $(q+1)-(2^f+1)-2^f(2^f-1)$ points of $\bar O_1$. Thus,
$N_1=2+\frac{q-2^{2f}}{2^f(2^{2f}-1)},$
and
$N_2=\frac{q^3-q}{|\bar L|}=\frac{q}{2^f} \cdot \frac{q-1}{2^{2f}-1} \cdot \frac{q+1}{w}.$
\end{itemize}
\end{proof}

\begin{lemma} \label{dihq-13}
Let $L \leq Aut(\cX_n)$ and let $\bar L =D_{2t} \times C_w$ where $D_{2t}$ is a dihedral group of order $2t$ with $t \mid (q-1)$ and $C_w=\bar L_Z$. Then
$$g_{\bar L}=\frac{q^2-qw-qt+wt+w-t-1}{4tw} \quad {\rm{and}} \quad N=\frac{q-1+3t}{2t}+\frac{q}{2} \cdot \frac{q-1}{t} \cdot \frac{q+1}{w}.$$
\end{lemma}

\begin{proof}
The genus $g_{\bar L}$ was already computed in \cite[Proposition 4.5]{DVMZ}.
Furthermore, $\bar L/\bar L_{Z}$ has an orbit of length $2$ given by the two fixed points on $\ell$ of its unique cyclic subgroup of order $t$, which is of type (B2) from Lemma \ref{classificazione}. Since $\bar L$ contains exactly $t$ involutions, and $\bar L$ acts transitively on the set of its involutions, $\bar L/\bar L_{Z}$ has another orbit in $\bar O_1$ of length $t$ given by their fixed points. In the set of the remaining points in $\bar O_1$, $\bar L/\bar L_{Z}$ acts with orbits of length $2t$. Hence
$N_1=2+\frac{(q+1)-2-t}{2t}=\frac{q-1+3t}{2t}$.
Since $\bar L$ acts semiregularly on $\bar O_2$ we have
$N_2=\frac{q^3-q}{|\bar L|}=\frac{q^3-q}{2t|\bar L_Z|}=\frac{q}{2} \cdot \frac{q-1}{t} \cdot \frac{q+1}{w}.$
\end{proof}

\begin{lemma} \label{A54}
Let $h$ be even and let $L \leq Aut(\cX_n)$ and let $\bar L=\mathbf{A}_5 \times C_w$, where $C_w=\bar L_Z$ and $h$ is even. Then
$$g_{\bar L}=\frac{(q+1)(q-w-16)+65w-48 \delta}{120w}, \quad
\makebox{where} \quad \delta= \begin{cases} w,  & \makebox{if $5 \mid (q-1)$,} \\ 0,  & \makebox{if $5 \mid (q+1)$ and $5 \nmid w$,} \\ q+1, & \makebox{if $5 \mid w$.} \end{cases}$$
Also,
$$N=\begin{cases}  \displaystyle 2+\frac{q-16}{60}+\frac{q}{4} \cdot \frac{q-1}{15} \cdot \frac{q+1}{w}, & \makebox{if $5 \mid (q-1)$,} \\ \\  \displaystyle 1+\frac{q-4}{60}+ \frac{q}{4} \cdot \frac{q-1}{3} \cdot \frac{q+1}{5w}, & \makebox{if $5 \mid (q+1)$ and $5\nmid w$,} \\ \\  \displaystyle 1+\frac{q-4}{60}+\frac{q+1}{w} \cdot \bigg(\frac{q^2-q-12}{60}+1 \bigg),  & \makebox{if $5 \mid w$.} \end{cases}$$
\end{lemma}

\begin{proof}
The genus $g_{\bar L}$ was already computed in \cite[Proposition 4.7]{DVMZ}. From its proof the action of $\bar L$ on $\bar O_1 \cup \bar O_2$ is depending on whether $5 \mid (q-1)$, or $5 \mid (q+1)$ but $5 \nmid |\bar L_Z|=w$ or $5 \mid |\bar L_Z|$.

\begin{itemize}
\item Let $5 \mid (q-1)$. $\bar L$ has $5$ Sylow $2$-subgroups of order $4$ each fixing a point on $\bar O_1$. This gives an orbit of $\bar L/\bar L_Z$ of length $5$. The remaining $12$ elements of order $5$ in $\bar L/\bar L_Z$ fix another point in $\bar O_1$ since they are of type $(B2)$ from Lemma \ref{classificazione}. The remaining $(q+1)-17$ points in $\bar O_1$ are fixed just by $\bar L_Z$ in $\bar L$, yielding that $\bar L$ acts on the remaining points in $\bar O_1$ with orbits of length $60$. Since no elements in $\bar L$ fix points in $\bar O_2$, $\bar L$ acts with long orbits on $\bar O_2$. This gives that
$N_1=2+\frac{q-16}{60},$
and
$N_2=\frac{q^3-q}{|\bar L|}=\frac{q}{4} \cdot \frac{q-1}{15} \cdot \frac{q+1}{w}.$
\item Let $5 \mid (q+1)$ but $5 \nmid |\bar L_Z|$. As before, $\bar L$ has one orbit of length $5$ in $\bar O_1$ given by the fixed points of its Sylow $2$-subgroups and no other elements in $\bar L$ apart from the ones in $\bar L_Z$ fix other points in $\bar O_1$ or in $\bar O_2$. Hence
$N_1=1+\frac{(q+1)-5}{60},$
and
$N_2=\frac{q^3-q}{|\bar L|}=\frac{q}{4} \cdot \frac{q-1}{3} \cdot \frac{q+1}{5w}.$
\item Let $5 \mid |\bar L_Z|$. The action on $\bar O_1$ is the same as for the previous case since again $5 \mid (q+1)$ and the elements of ${\rm{SL}}(2,q)$ are characterized just by their orders from Lemma \ref{classificazione} and the fact that ${\rm{SL}}(2,q)$ contains no elements of type (A). Now the 24 elements of order $5$ in $\bar L \cap {\rm{SL}}(2,q)$ are of type $(B1)$ and hence fix pointwise $6$ distinct self-polar triangles $T_1,\ldots,T_6$ having $\ell$ as a common side. From the proof of \cite[Proposition 4.7]{DVMZ}, $\bar L$ contains other $2$ subgroups (and hence $8$ elements) of order $5$ which are of type (A) and fix distinct sides of $T_1,\ldots,T_6$ respectively, and these sides are all different from $\ell$. This implies that $\bar O_2$ contains a set of $12(q+1)$ points, the points of intersection of the $2$ sides of $T_i$ with $\cH_q(\mathbb{F}_{q^2})$ for every $i=1,\ldots,6$, on which $\bar L$ acts with stabilizer of order $5$, and hence with orbits of length $12|\bar L_Z|$. Moreover $\bar L$ acts with long orbits on the remaining $q^3-q-12(q+1)$ points of $\bar O_2$. This gives that
$N_2=\frac{12(q+1)}{12|\bar L_Z|}+\frac{q^3-q-12(q+1)}{|\bar L|}=\frac{q+1}{w}\bigg( 1+\frac{q^2-q-12}{60}\bigg).$
\end{itemize}
\end{proof}

\begin{lemma} \label{A45}
Let $h$ be even and $L \leq Aut(\cX_n)$ and let $\bar L=\mathbf{A}_4 \times C_w$, where $C_w=\bar L_Z$ and $h$ is even. Then
$$g_{\bar L}=\frac{q^2-qw+4w-3q-4}{24w}, \quad \makebox{and} \quad N=\frac{q+20}{12}+\frac{q}{4} \cdot \frac{q-1}{3} \cdot \frac{q+1}{w}.$$
\end{lemma}

\begin{proof}
The genus $g_{\bar L}$ was already computed in \cite[Proposition 4.6]{DVMZ}. From its proof the action of $\bar L$ on $\bar O_1 \cup \bar O_2$ can be described as follows. Since $\bar L$ has a unique Sylow $2$-subgroup, $\bar L$ fixes a point in $\bar O_1$. Also the $4$ subgroups of $\bar L/\bar L_Z$ of order $3$ are of type (B2) from Lemma \ref{classificazione} and each of them fix another point on $\bar O_1$. This shows that $\bar L$ has a fixed point and one orbit of length $4$ on $\bar O_1$ and since the unique subgroup of $\bar L$ fixing at least another point in the remaining $(q+1)-1-4$ points of $\bar O_1$ is $\bar L_Z$, we get that $\bar L$ acts with orbits of length $12$ on the set of the remaining points of $\bar O_1$. Since $\bar L$ has no other elements fixing points in $\cH_q$, it acts semiregularly on $\bar O_2$. Hence
$N_1=2+\frac{(q+1)-1-4}{12}=\frac{q+20}{12},$
and
$N_2=\frac{q^3-q}{|\bar L|}=\frac{q^3-q}{12|\bar L_Z|}=\frac{q}{4} \cdot \frac{q-1}{3} \cdot \frac{q+1}{w}.$
\end{proof}

\begin{lemma} \label{elab1d7}
Let $L \leq Aut(\cX_n)$ and let $\bar L= (E_{2^f} \rtimes C_d) \times C_w$, where $f \leq h$, $E_{2^f}$ is elementary abelian of order $2^f$, $C_w=\bar L_Z$  and $d\mid \gcd(2^f-1,q-1)$. Then
$$g_{\bar L}=\frac{(q+1)(q-w-2^f)+w (2^f+1)}{2^{f+1}dw}, \quad {\rm{and}} \quad N=\frac{q-2^f}{2^f d}+2+\frac{q}{2^f} \cdot \frac{q^2-1}{dw}.$$
\end{lemma}

\begin{proof}
The computation of $g_{\bar L}$ was given in \cite[Proposition 4.8]{DVMZ}. Here $\bar L$ fixes exactly one point $\bar P_1 \in \bar O_1$, and the elements of order $d$ in $\bar L$ have exactly another fixed point in $\bar O_1$ since they are of type (B2) from Lemma \ref{classificazione}. This implies that $\bar L$ has an orbit of length $1$ in $\bar O_1$ and from the Orbit Stabilizer Theorem it has another orbit of length $|\bar L|/d|\bar L_Z|$, while it acts on the remaining $(q+1)-1-2^f$ elements in $\bar O_1$ with orbits of length $2^k d$. Also $\bar L$ acts with long orbits on $\bar O_2$ as no elements in $\bar L$ fix points in $\bar O_2$. This shows that
$N_1=2+\frac{(q+1)-1-2^k}{2^kd}=\frac{q-2^k}{2^k d}+2,$
while
$N_2=\frac{|O_2|}{|\bar L|}=\frac{q^3-q}{2^kd|\bar L_Z|}=\frac{q}{2^k} \cdot \frac{q^2-1}{dw}.$
\end{proof}


From Lemma \ref{sugroupsMlqeven} to complete this section we need to analyze the case $\bar L$ fixes a self-polar triangle $T=\{P_1,P_2,P_3\} \subseteq \mathbb{P}^2(\mathbb{F}_{q^2}) \setminus \cH_q(\mathbb{F}_{q^2})$ either fixing $P=P_1$ and acts transitively on $T \setminus \{P_1\}$ or fixing $T$ pointwise. We recall that the stabilizer of $T$ in $\PGU(3,q)$ is isomorphic to $(C_{q+1}\times C_{q+1})\rtimes \mathbf{S}_3$, where $C_{q+1}\times C_{q+1}$ fixes $T$ pointwise while $\mathbf{S}_3$ acts faithfully on $T$, see \cite{H} and \cite{M}.
For $\bar L \leq (C_{q+1}\times C_{q+1})\rtimes \mathbf{S}_3$ let $\bar L_T$ be the subgroup of $\bar L$ fixing $T$ pointwise. In Case 10 of Lemma \ref{sugroupsMlqeven} clearly $\bar L_T$ has index $2$ in $\bar L$.

\begin{proposition}{\rm(see \cite[Proposition 3.3]{DVMZ})}\label{index2}
Let $q$ be even. Let $T=\{P,P_1,P_2\}$ be a self-polar triangle in $\mathbb{P}^2(\mathbb{F}_{q^2}) \setminus \cH_q(\mathbb{F}_{q^2})$.
\begin{itemize}
\item[(i)] Let $a$, $w$, and $e$ be positive integers satisfying $e\mid(q+1)^2$, $w\mid(q+1)$, $a\mid w$, $aw\mid e$, $\frac{e}{a}\mid(q+1)$, and $\gcd\left(\frac{e}{aw},\frac{w}{a}\right)=1$.
Then there exists a subgroup $\bar L \leq ((C_{q+1}\times C_{q+1})\rtimes \mathbf{S}_3) \cap M_\ell$ of order $2e$ such that $|\bar L_T|=e$ and
\begin{equation}\label{genusindex2}
g_{\bar L}=\frac{(q+1)\left(q-2a-w-\frac{e}{w}+1\right)+3e}{4e}.
\end{equation}
\item[(ii)] Conversely, let $ \bar L \leq ((C_{q+1}\times C_{q+1})\rtimes \mathbf{S}_3) \cap M_\ell$ and $\bar L_T$ has index $2$ in $\bar L$. Define $e=|\bar L|/2$, $a$ to be the order of the subgroup of homologies of $\bar L$ with center $P_1$, which is equal to the order of the subgroup of homologies in $\bar L$ with center $P_2$, and $w=|\bar L_Z|$. Then $a,w$ and $e$ satisfy the numerical assumptions in point {\it (i)} and the genus $g_{\bar L}$ is given by Equation \eqref{genusindex2}.
\end{itemize}
\end{proposition}

\begin{remark}
Let $t$ and $w$ be divisors of $q+1$ and define $a=\gcd(t,w)$ and $e=tw$. It is not hard to see that these numbers satisfy the numerical conditions in {\it (i)}  Proposition \ref{index2}. Conversely any triple $a,w,e$ is obtained in this way by choosing $t=e/w$. Equation \eqref{genusindex2} reads,
$$g_{\bar L}=\frac{(q+1)(q-2\gcd(t,w)-w-t+1)+3tw}{4tw}.$$
\end{remark}

\begin{lemma} \label{isdihedral}
Let $\bar L \leq (C_{q+1} \times C_{q+1}) \rtimes \mathbf{S}_3$ such that $|\bar L|=2e=2tw$ and $\bar L_T$ has index $2$ in $\bar L$ and $w=|\bar L_Z|$. Then $\bar L / \bar L_{Z}$ is a dihedral group of order $2t$.
\end{lemma}

\begin{proof}
Since $\bar L_T$ has index $2$ in $\bar L$, $\bar L$ fixes a vertex $P \in T$ and acts transitively on $T \setminus \{P\}$. In particular $\bar L$ is a subgroup of $(C_{q+1} \times C_{q+1}) \rtimes C_2$. Hence $\bar L/\bar L_{Z} \leq ((C_{q+1} \times C_{q+1}) \rtimes C_2)/Z \cong C_{q+1} \rtimes C_2=D_{2(q+1)}$ which is a dihedral group of order $2(q+1)$. Since $\bar L/\bar L_Z$ is not abelian, it is a dihedral group of order $2t$.
\end{proof}

\begin{lemma} \label{dih2q8}
Let $\bar L$ be as in Proposition \ref{index2} and define $t=e/w$. Then
$$N=1+\frac{q-t+1}{2t}+\frac{(q+1)\gcd(w,t)}{tw}+\frac{(q+1)^2(q-2)}{2tw}.$$
\end{lemma}

\begin{proof}

From Lemma \ref{isdihedral} $\bar L/\bar L_Z$ is a dihedral group of order $2t$. Hence $\bar L/\bar L_Z$ contains exactly $t$ involutions and it acts on the subset of $\bar O_1$ given by their fixed points. From Proposition \cite[Proposition 3.3]{DVMZ} $\bar L$ acts with orbits of length $2t$ on the remaining points of $\bar O_1$. This shows that
$N_1=1+\frac{q+1-t}{2t}.$
The only elements in $\bar L$ that fix points in $\bar O_2$ are the ones contained in the two groups of homologies of order $\gcd(w,t)$ and they fix a set of $q+1$ points respectively on which $\bar L$ (from the Orbit Stabilizer Theorem) acts with orbits of length $2tw/\gcd(w,t)$. Hence $\bar L$ is semiregular on the remaining $q^3-q-2(q+1)$ points in $\bar O_2$. Hence,
$$N_2=\frac{2(q+1)\gcd(w,t)}{2tw}+\frac{(q^3-q)-2(q+1)}{2tw}=\frac{(q+1)\gcd(w,t)}{tw}+\frac{(q+1)^2(q-2)}{2tw}.$$
\end{proof}

%

\section{The case $q \equiv 1 \pmod 4$: determination of $\bar L$, $g_{\bar L}$ and $N$}



In this case another description of $M_\ell$ is given in \cite[Section 3]{MZq14}. Let $\beta$ be any involution of $M_\ell$ different from $\iota=[-1,0,1]$, for instance $\beta=[-1,0,-1]$; obviously, $\beta$ normalizes both $S_\ell$ and $Z$ and it does not commute with $S_\ell$ in general. Then

$$ M_\ell=(S_\ell \rtimes \langle\beta\rangle)\times Z_1 \cong ( {\rm{SL}}(2,q) \rtimes C_2)\times C_{\frac{q+1}{2}}\,. $$

With the notation of \cite[Section 9]{GKeven} the subgroup $S_\ell \rtimes \langle\beta\rangle$ is also denoted by ${\rm{SU}}^{\pm}(2,q)$, meaning that $S_\ell \rtimes\langle\beta\rangle$ consists of the elements of $M_\ell$ with determinant $1$ or $-1$; here, the determinant of an element $\alpha\in M_\ell$ is the determinant of the representative matrix of $\alpha$ having entry $1$ in the third row and column.

In this section, for any $\bar L \leq M_\ell$, we will use the notation $\bar L_{Z_1}=\bar L \cap Z_1$ and $w=|\bar L_{Z_1}|$.

The complete list of subgroups $\bar L \leq M_\ell$ is given in Lemma \ref{subqroupsMlq14}. However, in the cases in which $\bar L$ is tame the genus $g_L$ can be computed directly from $g_{\bar L}$ without having to compute  $N$. The values of $g_{\bar L}$ can be found in \cite{MZq14} and will not be reproduced here. We proceed by computing $N$ in the remaining, non-tame cases. That is to say, in Cases 1 with $p=3$, 7, 8, 12, and 15 from Lemma \ref{subqroupsMlq14}.

\begin{lemma} \label{sl25}
Let $p=3$ and $L \leq Aut(\cX_n)$ be such that $\bar L= {\rm{SL}}(2,5) \times C_w$, where $C_w=\bar L_{Z_1}$ and $q^2 \equiv 1 \pmod 5$. Then
$$ g_{\bar L}= \frac{(q+1)(q-21-2w)+140w-48s}{240w},$$
where
$$s=\begin{cases} 2w & \textrm{if}\quad5\mid(q-1), \\
0 & \textrm{if}\quad5\mid(q+1),\; 5\nmid w, \\
q+1 & \textrm{if}\quad5\mid w. \end{cases} $$
Also,
$$N=\begin{cases} \displaystyle \frac{q+99}{60}+\frac{q}{3} \cdot \frac{q-1}{5} \cdot \frac{q+1}{w}, & \textrm{if} \ 5 \mid (q-1), \\  \\ \displaystyle \frac{q+51}{60}+\frac{q(q-1)}{3} \cdot \frac{q+1}{5w}, & \textrm{if} \ 5 \mid (q+1) \ \textrm{and} \  5 \nmid w, \\ \\  \displaystyle \frac{q+51}{60}+\frac{q+1}{2w}+\frac{(q^2-q-12)(q+1)}{120w}, & \textrm{if} \ 5 \mid w.\end{cases}$$
\end{lemma}

\begin{proof}
The genus $g_{\bar L}$ is computed in \cite[Proposition 3.4]{MZq14}. According to its proof we can describe the short-orbits structure of $\bar L$ in its natural action on $\cH_q(\mathbb{F}_{q^2})$.
\begin{itemize}
\item Let $5 \mid (q-1)$. Since ${\rm{SL}}(2,5)$ contains 10 cyclic subgroups of order $p=3$, $\bar L$ contains a short orbit of length $10$ in $\bar O_1$ given by the corresponding fixed points. Also $\bar L$ contains $6$ cyclic subgroups $C_5$ and the order of the normalizer in ${\rm{SL}}(2,5)$ of $C_5$ is $20$. From Lemma \ref{classificazione} each element in $C_5$ fixes exactly $2$ points in $\bar O_1$ and the set of the corresponding fixed points is disjoint from the set of $10$ points counted before, since the normalizer of an element of order $3$ has order coprime with $5$. Hence elements of order $4$ and of order $5$ fixes distinct points on $\bar O_1$ implying that there is another short orbit of ${\rm{SL}}(2,5)$ on $\bar O_1$ of length $12$ because the corresponding stabilizer in ${\rm{SL}}(2,5)$ has order $10$. ${\rm{SL}}(2,5)$ acts with orbits of length $60$ elsewhere since its unique involution is central and hence fixes $\bar O_1$ pointwise. Hence
$N_1=1+1+\frac{q+1-10-12}{60}=\frac{q+99}{60},$
and
$N_2=\frac{q^3-q}{|\bar L|}=\frac{q^3-q}{120|\bar L_{Z_1}|}=\frac{q}{3} \cdot \frac{q-1}{5} \cdot \frac{q+1}{w}.$

\item Let $5 \mid (q+1)$ but $5 \nmid w$. As before, since ${\rm{SL}}(2,5)$ contains 10 cyclic groups of order $p=3$, $\bar L$ contains a short orbit of length $10$ in $\bar O_1$ given by the corresponding fixed points. Now the cyclic groups of order $5$ acts semiregularly on $\cH_q(\mathbb{F}_{q^2})$ as they are of type (B1) from Lemma \ref{classificazione}. Hence $\bar L$ has exactly one short orbit on $\bar O_1$ of length $10$ and acts with orbits of length $60$ elsewhere as its unique involution fixes $\bar O_1$ pointwise. ${\rm{SL}}(2,5)$ acts semiregularly on $\bar O_2$. Hence
$N_1=1+\frac{q+1-10}{60}=\frac{q+51}{60},$
and
$N_2=\frac{q^3-q}{|\bar L|}=\frac{q^3-q}{120|\bar L_{Z_1}|}=\frac{q(q-1)}{3} \cdot \frac{q+1}{5w}.$

\item Let $5 \mid w$. As before, since ${\rm{SL}}(2,5)$ contains 10 cyclic groups of order $p=3$, $\bar L$ contains a short orbit of length $10$ in $\bar O_1$ given by the corresponding fixed points. Now the cyclic groups of order $5$ acts semiregularly on $\cH_q(\mathbb{F}_{q^2})$ as they are of type (B1) from Lemma \ref{classificazione}. Hence $\bar L$ has exactly one short orbit on $\bar O_1$ of length $10$ and acts with orbits of length $60$ elsewhere as its unique involution fixes $\bar O_1$ pointwise. ${\rm{SL}}(2,5)$ acts semiregularly on $\bar O_2$. From the proof of \cite[Proposition 3,4]{MZq14}, $\bar L$ contains $12$ cyclic subgroups of order $5$ which are of type (A), that is, that fix $q+1$ distinct points in $\bar O_2$. This implies that $O_2$ contains a set of $12(q+1)$ points on which $\bar L$ acts with stabilizer of order $5$, and hence with orbits of length $24|\bar L_Z|$ and acts with long orbits elsewhere. Hence,
$N_1=1+\frac{q+1-10}{60}=\frac{q+51}{60},$
and
$N_2=\frac{12(q+1)}{24|\bar L_{Z_1}|}+\frac{q^3-q-12(q+1)}{120|\bar L_{Z_1}|}=\frac{q+1}{2w}+\frac{(q^2-q-12)(q+1)}{120w}.$
\end{itemize}
\end{proof}

\begin{lemma} \label{sottosl}
Let $L \leq Aut(\cX_n)$ and let $\bar L= {\rm{SL}}(2,p^k) \times C_w$ where $k \mid h$. Let $r=h/k$. Then
$$g_{\bar L}=1+\frac{q^2-q-2-\Delta}{2p^k(p^{2k}-1)w},$$
where
$$\Delta=(p^{2k}-1)(q+2)+p^{2k}-1+q+1+p^k(p^k+1)(p^k-3)w+p^k(p^k-1)^2(\gcd(r,2)-1)+2(p^{2k}-1)(w-1)$$
$$+2(w-1)(q+1)+p^k(p^k-1)^2(w-1)(\gcd(r,2)-1)+(\gcd(w,p^k+1)-1)p^k(p^k-1)(q+1)(2-\gcd(r,2)).$$
Also,
$$N=\begin{cases} \displaystyle 2+\frac{2(q-p^{2k})}{p^k(p^k-1)(p^k+1)}+\frac{q}{p^k} \cdot \frac{q-1}{(p^{2k}-1)} \cdot \frac{q+1}{w}, & \textrm{if} \ r \ \textrm{is even}, \\ \\ \displaystyle 1+\frac{2(q-p^k)}{p^k(p^{2k}-1)}+\frac{(q+1)\gcd(p^k+1,w)}{(p^k+1)w}+\frac{(q^2-q-p^k(p^k-1))(q+1)}{p^k(p^k-1)(p^k+1)w}, & \textrm{if} \ r \ \textrm{is odd}. \end{cases}$$
\end{lemma}

\begin{proof}
The genus $g_{\bar L}$ was computed in \cite[Proposition 3.9]{MZq14}. From its proof we can describe the short-orbits structure of the group $\bar L$. We will distinguish two cases.
\begin{itemize}
\item Assume that $r$ is even. Then $p^{2k}-1$ divides $q-1$ and hence every element of order dividing $p^{2k}-1$ in ${\rm{SL}}(2,p^k)$ is of type (B2) from Lemma \ref{classificazione}. Since ${\rm{SL}}(2,p^k)$ contains exactly $p^k+1$ Sylow $p$-subgroups, $\bar L$ has a short orbit of length $p^k+1$ contained in $\bar O_1$. Since every cyclic subgroup of ${\rm{SL}}(2,p^k)$ is of type (B2), and ${\rm{SL}}(2,p^k)$ contains exactly $p^k(p^k-1)/2$ subgroups of order $p^k+1$, we get that $\bar O_1$ contains another short orbit of $\bar L$ given by the corresponding $p^k(p^k-1)$ fixed points. No other elements in $\bar L$ other than the central involution and the elements in $\bar L_{Z_1}$ fix other points in $\bar O_1$ and hence $\bar L$ acts with orbits of length $p^k(p^k+1)(p^k-1)/2$ on the remaining points in $\bar O_1$. Also $\bar L$ acts semiregularly on $\bar O_2$. Hence
$N_1=1+1+2\frac{q+1-(p^k+1)-p^k(p^k-1)}{p^k(p^k-1)(p^k+1)}=2+\frac{2(q-p^{2k})}{p^k(p^k-1)(p^k+1)}$
and
$N_2=\frac{q^3-q}{|\bar L|}=\frac{q}{p^k} \cdot \frac{q-1}{(p^{2k}-1)} \cdot \frac{q+1}{w}.$

\item Let $r$ be odd. Then $p^k+1$ divides $q+1$. Since ${\rm{SL}}(2,p^k)$ contains exactly $p^k+1$ Sylow $p$-subgroups, $\bar L$ has a short orbit of length $p^k+1$ contained in $\bar O_1$. No other elements other than the unique involution of ${\rm{SL}}(2,p^k)$ (which is central) and $\bar L_{Z_1}$ fix other elements in $\bar O_1$. This implies that $\bar L$ acts with orbits of length $p^k(p^k+1)(p^k-1)/2$ on the remaining points in $\bar O_1$.
Since $p^k+1$ divides $q+1$ it might be that $\gcd(p^k+1,|\bar L_Z|)$ is not trivial. From \cite[Proposition 3.9]{MZq14} if this happens then $\bar O_2$ contains a subset of $p^k(p^k-1)(q+1)$ points whose stabilizer in $\bar L$ has order $\gcd(p^k+1,w)$, and hence on which $\bar L$ acts with orbits of length $|\bar L|/\gcd(p^f+1,w)$. Also $\bar L$ acts with long orbits on the remaining $q^3-q-p^f(p^f-1)(q+1)$ points in $\bar O_2$.
Hence
$N_1=1+2\frac{q+1-(p^k+1)}{p^k(p^{2k}-1)}=1+\frac{2(q-p^k)}{p^k(p^{2k}-1)},$
and
$N_2=\frac{p^k(p^k-1)(q+1)\gcd(p^k+1,|\bar L_{Z_1}|)}{|\bar L|}+\frac{q^3-q-p^k(p^k-1)(q+1)}{|\bar L|}=\frac{(q+1)\gcd(p^k+1,w)}{(p^k+1)w}+\frac{(q^2-q-p^k(p^k-1))(q+1)}{p^k(p^k-1)(p^k+1)w}.$
\end{itemize}
\end{proof}

\begin{lemma} \label{TL}
Let $L \leq Aut(\cX_n)$ and let $\bar L={\rm{TL}}(2,p^k) \times C_w$ where $k \mid h$ and $r=h/k$ is even. Then
$$g_{\bar L}=1+\frac{q^2-q-2-\Delta}{4p^k(p^{2k}-1)w},$$
where
$$\Delta=(p^{2k}-1)(q+2)+p^{2k}-1+q+1+p^k(p^k+1)(p^k-3)w+p^k(p^k-1)^2+2(p^{2k}-1)(w-1)$$
$$+2(w-1)(q+1)+p^k(p^k-1)^2(w-1)+2p^k(p^{2k}-1)w,$$
and
$$N=2+\frac{q-p^{2k}}{p^k(p^{2k}-1)}+\frac{q}{p^k} \cdot \frac{q-1}{(p^{2k}-1)} \cdot \frac{q+1}{2w}.$$
\end{lemma}

\begin{proof}
The genus $g_{\bar L}$ is computed in \cite[Proposition 3.10]{MZq14} and from its proof the short orbit structure of $\bar L$ on $\cH_q(\mathbb{F}_{q^2})$ can be described.
The short orbit structure of the subgroup ${\rm{SL}}(2,p^k) \times \bar L_{Z_1}$ was already described in the proof of Lemma \ref{sottosl} and every element in ${\rm{TL}}(2,p^k) \setminus {\rm{SL}}(2,p^k)$ is of type (B2) from \cite[Proposition 4.4]{MZ}. Hence the action of $\bar L$ on $\bar O_2$ is semiregular as ${\rm{SL}}(2,p^k)$ acts semiregularly on $\bar O_2$ and elements of type (B2) fix just points on $\bar O_1$. Since ${\rm{SL}}(2,p^k)$ has exactly two short orbits of distinct lengths on $\bar O_1$ then they are also short orbits of ${\rm{TL}}(2,p^k)$.
%

Since ${\rm{SL}}(2,p^k)$ has index $2$ in ${\rm{TL}}(2,p^k)$ either there exists a point $R$ in an orbit of cardinality $p^k(p^{2k}-1)/2$ of ${\rm{SL}}(2,p^k)$ with a stabilizer of order $4$, or all remaining orbits under the action of ${\rm{TL}}(2,p^k)$ have cardinality $p^k(p^{2k}-1).$
From the proof of \cite[Proposition 4.4]{MZ}, an element $\alpha \in {\rm{TL}}(2,p^k) \setminus {\rm{SL}}(2,p^k)$ has order at least $5$. Therefore $\alpha$ cannot occur in the stabilizer of $R$ implying that the order of such a stabilizer remains two.
Thus,
$N_1=2+\frac{q+1-(p^k+1)-p^k(p^k-1)}{p^k(p^{2k}-1)}=2+\frac{q-p^{2k}}{p^k(p^{2k}-1)},$
and
$N_2=\frac{q^3-q}{|\bar L|}=\frac{q}{p^k} \cdot \frac{q-1}{(p^{2k}-1)} \cdot \frac{q+1}{2w}.$
\end{proof}




\begin{lemma} \label{sottosupm}
Let $L \leq Aut(\cX_n)$ and let $\bar L={\rm{SU}}^{\pm}(2,p^k) \times C_w \cong ({\rm{SL}}(2,p^k) \rtimes C_2)\times C_w$, where $k \mid h$, $h/k$ is odd and $C_w=\bar L_{Z_1}$. Then
$$g_{\bar L}=1+\frac{q^2-q-2-\Delta}{4p^k(p^{2k}-1)w},$$
where
$$\Delta=(q+1)+p^k(p^k+1)(p^k-3)+(p^{2k}-1)(q+3)+p^k(p^k-1)(q+1)+p^k(p^{2k}-1)+(2w-2)(q+1)$$
$$+2(p^{2k}-1)(w-1)+2p^k(p^k+1)(p^k-2)(w-1)+2p^k(p^k-1)(q+1)(\gcd(p^k+1,w)-1).$$
Moreover
$$N=1+\frac{q-p^k}{p^k(p^k-1)(p^k+1)}+\frac{(q+1)\gcd(p^k+1,w)}{w}+\frac{[q^2-q-p^k(p^k-1)]}{p^k(p^k+1)(p^k-1)} \cdot \frac{(q+1)}{2w}.$$
\end{lemma}

\begin{proof}
The genus $g_{\bar L}$ is computed in \cite[Proposition 3.15]{MZq14}. From its proof and the proof of Lemma \ref{sottosl} the short orbits structure of $\bar L$ on $\cH_q$ can be described.

The number of Sylow $p$-subgroups of ${\rm{SU}}^{\pm}(2,q)$ is the same as the number of Sylow $p$-subgroups of ${\rm{SL}}(2,p^k)$, hence $\bar L$ has a short orbit $\theta$ of length $p^k+1$ contained in $\bar O_1$. The stabilizer of a point $\bar R \in \bar O_1 \setminus \theta$ in ${\rm{SL}}(2,p^k)$ has order $2$. The order of the stabilizer of $\bar R$ in ${\rm{SU}}^{\pm}(2,p^k)$ can grow if and only if there exists an element $\alpha \in {\rm{SU}}^{\pm}(2,p^k) \setminus {\rm{SL}}(2,p^k)$ fixing $\bar R$.
If $\alpha$ is of type (B2) or $(E)$ then $\alpha^2$, which is in ${\rm{SL}}(2,p^k)$, fixes exactly two points in $\theta$. If $\alpha$ is of type (B1) then it acts without fixed points on $\bar O_1$. If $\alpha$ is of type (A) then it has no fixed points on $\bar O_1$ unless $\alpha$ is central and hence in ${\rm{SL}}(2,p^k)$, a contradiction. This shows that ${\rm{SU}}^{\pm}(2,p^k)$ acts with orbits of length $|{\rm{SU}}^{\pm}(2,q)|/2=|{\rm{SL}}(2,p^k)|$ on $\bar O_1 \setminus \theta$. Since $\bar L_{Z_1}$ acts trivially on $\bar O_1$ the length of the orbits of $\bar L$ on $\bar O_1$ is the same as the length of the orbits of ${\rm{SU}}^{\pm}(2,p^k)$. Hence
$N_1=1+\frac{q+1-(p^k+1)}{|{\rm{SL}}(2,p^k)|}=1+\frac{q-p^k}{p^k(p^k-1)(p^k+1)}.$

From the proof of \cite[Proposition 3.15]{MZq14} we have that, even though ${\rm{SL}}(2,p^k)$ acts semiregularly on $\bar O_2$, ${\rm{SU}}^{\pm}(2,p^k)$ contains $p^k(p^k-1)$ elements of type (A) and order $2$ fixing $q+1$ distinct points on $\bar O_2$. Also, combining this with the proof of Lemma \ref{sottosl}, this set of points can have a non-trivial stabilizer in ${\rm{SL}}(2,p^k) \times \bar L_{Z_1}$ of order $\gcd(p^k+1,w)$ given by elements of type $(A)$ having the same axis and center as the described elements of order $2$ in ${\rm{SU}}^{\pm}(2,p^k)$. Hence the stabilizer has order $2\gcd(p^k+1,w)$ in this case.

Therefore $\bar O_2$ contains a set of $p^k(p^k-1)(q+1)$ points on which $\bar L$ acts with orbits of length
$|\bar L|/2\gcd(p^k+1,w)=|{\rm{SU}}^{\pm}(2,p^k)|w/2\gcd(p^k+1,w)=|{\rm{SL}}(2,p^k)|w/\gcd(p^k+1,w)$. Moreover, $\bar L$ acts semiregularly elsewhere in $\bar O_2$.
Hence,

$N_2=\frac{p^k(p^k-1)(q+1)\gcd(p^k+1,w)}{p^k(p^k-1)(p^k+1)w}+\frac{q^3-q-p^k(p^k-1)(q+1)}{|\bar L|}=\frac{(q+1)\gcd(p^k+1,w)}{w}+\frac{[q^2-q-p^k(p^k-1)]}{p^k(p^k+1)(p^k-1)} \cdot \frac{(q+1)}{2w}.$
\end{proof}

At this point we are left with Case 15 from Lemma \ref{subqroupsMlq14}. In this case the genus $g_{\bar L}$ is computed in \cite{GSX} and \cite{BMXY}.

The following lemma is a consequence of Corollary 4.5 in \cite{GSX} and the subsequent observations.

\begin{lemma}{\rm{\cite[Corollary 4.5]{GSX}}}
Let $q=p^h$, $m \mid (q^2-1)$ and $u \leq h$. Let $R \in \cH_q(\mathbb{F}_{q^2})$. If $\bar L$ is a subgroup of $\PGU(3,q)$ fixing $R$, such that $|\bar L|=m p^u$ and $\bar L$ has a (unique) elementary abelian Sylow $p$-subgroup, then
$$g_{\bar L}=\frac{1}{2m}(q+1-d)(p^{h-u}-1),$$
where $d=\gcd(q+1,m)$.
Finally,
$$N=\frac{q(q^2-1)}{p^u m}+2+\frac{d(q-p^u)}{p^u m}.$$
\end{lemma}

\begin{proof}

Since $\bar L$ fixes $\bar R$ it has the short orbit $\{\bar R\}$ of length $1$ in $\bar O_1$. From the proof of \cite[Theorem 4.4]{GSX} one can derive directly that the sum of the number of fixed points of the elements in $\bar L$ when acting on the remaining $q^3$ points in $\bar O_1 \cup \bar O_2$ is equal to $q^3+|\bar L|+d(q-p^u)-q.$ Using Burnside's lemma we obtain
$$N=1+\frac{q^3-q+|\bar L|+d(q-p^u)}{|\bar L|}=\frac{q(q^2-1)}{p^u m}+2+\frac{d(q-p^u)}{p^u m}.$$

\end{proof}

\begin{remark}
Note that it is not true that for any $m \mid (q^2-1)$ and $u \leq h$, there exists a subgroup of $\PGU(3,q)$ fixing $R$ of order $m p^u.$ A full classification of the possible subgroups has been carried out in \cite{BMXY}. However, the genera one would obtain using such subgroups and Theorem \ref{thm:genusrelations} are already obtained in \cite{ABB}.
\end{remark}

\section{New genera for maximal function fields}

In this section, we give several examples of new genera of maximal function fields. New means that for the indicated finite fields, these genera cannot be constructed using methods or results from \cite{AQ,ABB,BMXY,CKT,DVMZ,DO,FG,GSX,GKT,GMQZ,GOS,MX,MZ,MZq14}. Note that for $q \equiv 1 \pmod{4}$, a complete list of subgroups of the automorphism group of the Hermitian curve is known as well as the corresponding genera. Therefore all new genera correspond to function fields that cannot be Galois covered by the Hermitian curve.

$$
{\renewcommand{\arraystretch}{1.5}
\begin{array}{l|l}
\mathbb{F}_{q^{2n}} & \makebox{new genera}\\
\hline
\mathbb{F}_{2^{20}} & 72, 200, 204, 302, 702, 1532, 3572\\
\mathbb{F}_{2^{28}} & 140, 492,560,1962,57332\\
\mathbb{F}_{5^{6}}  & 80,160,482\\
\mathbb{F}_{5^{10}} & 2340,4160,4680,6241,12484\\
\mathbb{F}_{5^{12}} & 19500\\
\mathbb{F}_{5^{14}} & 337,338,676,2016,3584,4032,5377,5378,10756,58590,117180,156241,156242,312484\\
\mathbb{F}_{9^{14}} & 84, 210, 350, 420, 448, 658, 700, 1122, 1124, 1402, 2248, 49476, 123690,
206150, 247380, 263872, 387562,\\
 & 412300, 659682, 659684, 824602, 1319368,
1434888, 3587220, 5978700, 7174440, 7652736, 11239956,\\
& 11957400, 19131842, 19131844, 23914802, 38263688 \\
\mathbb{F}_{{13}^{10}}& 240, 245, 281, 490, 738, 843, 846, 983, 1476, 1692, 1970, 57840, 59045, 67481,
118090, 177138, 202443,\\
 & 202446, 236183, 354276, 404892, 472370, 636480, 649740,
742561, 1299480, 1949223, 2227683,\\
& 2227686, 2598963, 3898446, 4455372, 5197930\\
\end{array}
}$$

\section*{Acknowledgments}

The first author would like to acknowledge the support from The Danish Council for Independent Research (DFF-FNU) for the project \emph{Correcting on a Curve}, Grant No.~8021-00030B.
The second author would like to thank the Italian Ministry MIUR, Strutture Geometriche, Combinatoria e loro
Applicazioni, Prin 2012 prot.~2012XZE22K and GNSAGA of the Italian INDAM.

\vspace{1ex}
\noindent
Peter Beelen

\vspace{.5ex}
\noindent
Technical University of Denmark,\\
Department of Applied Mathematics and Computer Science,\\
Matematiktorvet 303B,\\
2800 Kgs. Lyngby,\\
Denmark,\\
pabe@dtu.dk\\

\vspace{1ex}
\noindent
Maria Montanucci

\vspace{.5ex}
\noindent
Universit\'a degli Studi della Basilicata,\\
Dipartimento di Matematica, Informatica ed Economia,\\
Campus di Macchia Romana,\\
Viale dell' Ateneo Lucano 10,\\
85100 Potenza,\\
Italy,\\
maria.montanucci@unibas.it


\begin{thebibliography}{99}

\bibitem{AQ} M. Abd\'on and L. Quoos, {\it On the genera of subfields of the Hermitian function field}, Finite Fields Appl. {\bf 10}, 271--284, (2004).

\bibitem{ABB} N. Anbar, A. Bassa and P. Beelen,  {\it A complete characterization of Galois subfields of the generalized Giulietti-Korchm\'aros function field}, Finite Fields Appl. {\bf 48}, 318--330, (2017).

\bibitem{BMXY} A. Bassa, L. Ma, C. Xing and S.L. Yeo, {\it Towards a characterization of subfields of the Deligne-Lusztig function fields}, J. Combin. Theory Ser. A {\bf 120} (7), 1351--1371, (2013).

\bibitem{BM} P. Beelen and M. Montanucci, {\it A new family of maximal curves}, Journal of the London Math. Soc., appeared online. DOI: 10.1112/jlms.12144.

\bibitem{CKT} A. Cossidente, G. Korchm\'aros and F. Torres, {\it Curves of large genus covered by the Hermitian curve}, Comm. Algebra \textbf{28} (10), 4707--4728, (2000).

\bibitem{DVMZ} F. Dalla Volta, M. Montanucci and G. Zini, {\it On the classification problem for the genera of quotients of the Hermitian curve}, preprint, arXiv:1805.09118.

\bibitem{DO} Y. Dani\c{s}man and M. \"Ozdemir, {\it On the genus spectrum of maximal curves over finite fields}, J. Discr. Math. Sc. and Crypt. {\bf 18} (5), 513--529, (2015).


\bibitem{FG}  S. Fanali and M. Giulietti, {\it Quotient curves of the GK curve}, Adv. Geom. {\bf 12} (2), 239--268, (2012).

\bibitem{GGS} A. Garcia, C. G\"uneri and H. Stichtenoth, {\it  A generalization of the Giulietti-Korchm\'aros maximal curve}, Adv. Geom. {\bf 10} (3), 427--434, (2010).

\bibitem{GSX}  A. Garcia, H. Stichtenoth and C.P. Xing,  {\it On subfields of the Hermitian function field}, Compositio Math. {\bf 120}, 137--170, (2000).

\bibitem{GK} M. Giulietti and G. Korchm\'aros, {\it A new family of maximal curves over a finite field}, Math. Ann. {\bf 343}, 229--245, (2009).

\bibitem{GKeven} M. Giulietti and G. Korchm\'aros, {\it Algebraic curves with many automorphisms}, preprint, arXiv:1702.08812.

\bibitem{GKT} M. Giulietti, G. Korchm\'aros and F. Torres, {\it Quotient curves of the Suzuki curve}, Acta Arithmetica {\bf 122} (3), 245--274, (2006).

\bibitem{GMQZ} M. Giulietti, M. Montanucci, L. Quoos and G. Zini, {\it On some Galois covers of the Suzuki and Ree curves}, Journal of Number Theory {\bf 189}, 220--254, (2018).

\bibitem{GOS}  C. G\"uneri, M. \"Ozdemir and H. Stichtenoth, {\it The automorphism group of the generalized Giulietti–Korchm\'aros function field}, Adv. Geom. {\bf 13}, 369--380,  (2013).

\bibitem{H} R.W. Hartley, {\it Determination of the ternary collineation groups whose coefficients lie in the $GF(2^n)$}, Ann. of Math. Second Series {\bf 27} (2), 140--158, (1925).

\bibitem{HKT} J.W.P.~Hirschfeld, G.~Korchm\'aros and F.~Torres, \emph{Algebraic Curves over a Finite Field,} \textit{Princeton Series in Applied Mathematics}, Princeton, (2008).



\bibitem{Lach} G. Lachaud, {\it Sommes d'Eisenstein et nombre de points de certaines courbes alg\'ebriques sur les corps finis}, C. R. Acad. Sci. Paris S\'er. I Math. {\bf 305} (16), 729--732, (1987).

\bibitem{MX} L. Ma and C. Xing, {\it On subfields of the Hermitian function fields involving the involution automorphism}, preprint, arXiv:1707.07314.

\bibitem{M} H.H. Mitchell, {\it Determination of the ordinary and modular ternary linear groups}, Trans. Amer. Math. Soc. {\bf 12} (2), 207--242, (1911).

\bibitem{MZrs} M. Montanucci and G. Zini, {\it Some Ree and Suzuki curves are not Galois covered by the Hermitian curve}, Finite Fields Appl. {\bf 48}, 175--195, (2017).

\bibitem{MZ} M. Montanucci and G. Zini, {\it On the spectrum of genera of quotients of the Hermitian curve}, Comm. Algebra {\bf 46} (11), 4739--4776, (2018).

\bibitem{MZq14}  M. Montanucci and G. Zini, {\it The complete list of genera of quotients of the $\mathbb{F}_{q^2}$-maximal Hermitian curve for $q \equiv 1 \pmod 4$}, preprint, arXiv:1806.04546.




\end{thebibliography}
\end{document}